\documentclass[10pt]{amsart}
\usepackage{epsfig}
\usepackage{color}
\usepackage{amsmath, amssymb, euscript,  enumerate}
\usepackage{pxfonts}
\usepackage{graphicx}
\usepackage{subfigure}
\usepackage{ulem}

\newcommand{\R}{{\mathbb R}}
\newcommand{\Z}{{\mathbb Z}}
\newcommand{\cM}{{\mathcal M}}
\newcommand{\e}{\varepsilon}
\newcommand{\vp}{\varphi}
\newcommand{\osc}{\operatornamewithlimits{osc}}

\newcommand{\ra}{\rightarrow}
\newcommand{\La}{\triangle}

\newtheorem{thm}[subsubsection]{Theorem}
\newtheorem{lem}[subsubsection]{Lemma}
\newtheorem{definition}[subsubsection]{Definition}
\newtheorem{cor}[subsubsection]{Corollary}
\newtheorem{remark}[subsubsection]{Remark}
\newtheorem{prop}[subsubsection]{Proposition}

\numberwithin{equation}{subsection}

\title[Homogenization of soft inclusions]{The viscosity Method for the Homogenization of soft inclusions }
\author{Ki-ahm Lee}
\address{Seoul National University, Seoul 151-747, Korea}
\email{kiahm@snu.ac.kr}
\author{Minha Yoo}
\address{Seoul National University, Seoul 151-747, Korea }
\email{minha00@snu.ac.kr}
\begin{document}

\begin{abstract}
In this paper, we consider periodic soft inclusions $T_{\e}$ with periodicity $\e$, where the solution , $u_{\e}$, satisfies  semi-linear elliptic equations of non-divergence in $\Omega_{\e}=\Omega\setminus \overline{T}_\e$ with a Neumann data on $\partial T^{\mathfrak a}$ . The difficulty lies in the non-divergence structure of the operator where the standard energy method based on the divergence theorem can not be applied. The main object is developing a viscosity method to find the homogenized equation satisfied
by the limit of $u_{\e}$, called as $u$, as $\e$ approaches to zero.
We introduce the concept of a compatibility condition between the equation and the Neumann condition on the boundary  for the existence of uniformly bounded  periodic first correctors. The concept of second corrector has been developed to show the limit, $u$, is the viscosity solution of a homogenized equation.
\end{abstract}
\maketitle

\section{Introduction}
\subsection{} 
Let $\Omega$ be a bounded and connected domain in $\mathbb{R}^n$ with a smooth boundary. We are going to define a perforated domain $\Omega_{\e}$ by removing a $\e$-periodic balls out of $\Omega$.
For each $m \in \mathbb{Z}^n$, let  $B_a(m)$ be a ball with center $m$ and radius $0 < \mathfrak{a} <\frac{1}{2}$.
Let
\begin{equation*} \begin{split}
T^{\mathfrak a} &:= \bigcup_{m \in \mathbb{Z}^n} B_{\mathfrak a}(m), \\
T_{\e} &:= \e T^{\mathfrak a},
\end{split} \end{equation*}
and
$$ \Omega_\e = \Omega \setminus \overline{T}_\e .$$

The homogenization of partial differential equations in a perforated domain with Dirichlet or Neumann boundary value has been studied by many authors. Please refer \cite{JKO} and \cite{CL}  for details.

 In this paper, we will consider the generalization of the following soft inclusions where the diffusion coefficients are zero on the holes:
\begin{equation}\label{eqn-main-0} \begin{cases}
\triangle u_\e = f(x) &\text{ in } \Omega_\e\\
\displaystyle \frac{\partial u_\e}{\partial \nu} = 0 &\text{ on } \partial T_\e\\
u_\e = \vp(x) &\text{ on } \partial \Omega \setminus \partial  T_\e.
\end{cases} \end{equation}

In  \cite{JKO}, they show that $u_\e$ converges to $u_0$ weakly in $H^1(\Omega)$(strongly in $L^2(\Omega)$),  and that the limit $u_0$ satisfies
\begin{equation} \label{eqn-rst-0} \begin{cases}
\overline{a}_{ij} D_{ij} u_0 = \theta f(x) &\text{ in } \Omega \\
u_0 = \vp(x) &\text{ on } \partial \Omega
\end{cases} \end{equation}
for some constant matrix $(\overline{a}_{ij})$ where $\theta = \displaystyle\int_{[0,1]^n} \chi_{T^{\mathfrak a}}$. Their method relies on the energy estimates and compensated compactness to pass the limit in the weak formulation. Such energy method cannot be applicable to nonlinear equations of  non-divergence type since the solutions may have  different order of energies, \cite{CL}, even though they satisfy equations in the same class.

In this paper, we are going to develop a viscosity method
to find the homogenization process of the following semi-linear equation of non-divergence type:
\begin{equation}\label{eqn-main-1} \tag{$P_{\e}$}
\begin{cases}
L\left(D^2u_\e,u_\e,x,\frac{x}{\e}\right) = f\left(x,\frac{x}{\e}\right) &\text{ in } \Omega_\e\\
G\left( Du_\e(x), \frac{x}{\e} \right) = 0 &\text{ on } \partial T_\e \\
u_\e = \vp(x) &\text{ on } \partial \Omega \setminus \partial  T_\e
\end{cases} \end{equation}
where $L\left(D^2u_\e,u_\e,x,\frac{x}{\e}\right) = a_{ij}\left(\frac{x}{\e}\right) D_{ij} u_\e + c\left(u_\e, x, \frac{x}{\e}\right)$, $G\left( Du_\e(x), \frac{x}{\e} \right)=b^i\left( \frac{x}{\e} \right) D_i u_\e(x)$ and $\vp \in \mathcal{C}^2(\overline{\Omega})$. And the equation satisfies the following conditions.

{\bf Conditions I:}
\begin{enumerate}[(i)]
\item \label{con-unif-elliptic}
$L$ is  uniformly elliptic:  there are positive constants $0<\lambda\leq \Lambda<\infty$  such that
\begin{equation*}
\lambda |\xi|^2 \le a_{ij}(y) \xi^i \xi^j \le \Lambda |\xi|^2
\end{equation*}
for all $\xi \in \R^n$ and for all $y \in \R^n$.
\item \label{con-oblique}
$b^i(y)$ satisfies the uniform oblique condition:  there is a uniform constant $\mu$ satisfying $0<\mu \le b(y) \cdot\nu < \infty$. And, for the convenience, we also assume  $\|b\|_{L^{\infty}}\leq 1$.
\item \label{con-y-periodicity}
$a^{ij}(y)$, $b^i(y)$, $c(r,x,y)$ and $f(x,y)$ are periodic in $y$-variable: for every $m \in \Z^n$, we have
\begin{equation*}
a^{ij}(y) = a^{ij}(y+m), \, b^i(y) = b^i(y+m),\, c(r,x,y) = c(r,x,y+m) \text{ and } f(x,y) = f(x,y+m).
\end{equation*}
\item \label{con-c-decr}
$c(0,x,y)=0$ and $c(r,x,y)$ is non-increasing with $r$ variables.
\item\label{con-alpha}
\begin{equation*} \begin{aligned}
\| a_{ij} \|_{C^\alpha(\R^n \setminus T^{\mathfrak a})} + \| b^{i}(\mathfrak{a}~\cdot) \|_{C^{1,\alpha}\left(\frac{1}{a} (\R^n \setminus T^{\mathfrak a})\right)} &\le \Lambda \text{ and,}\\
\| c(r,x,\cdot) \|_{C^\alpha(\R^n \setminus T^{\mathfrak a})} + \| f(x,\cdot) \|_{C^\alpha(\R^n \setminus T^{\mathfrak a})} &\text{ is bounded}
\end{aligned}  \end{equation*}
for every $(r,x) \in \R \times \R^n$.
\item \label{con-conti}
$f$ and $c$ are uniformly continuous with respect to $y$ variable. That is, for any given $x_0 \in \overline{\Omega}$ and $r_0 \in \R$,
\begin{equation*}
\lim_{x \ra x_0} \sup_{y \in \R^n} \left| f(x,y) -f(x_0,y) \right| = 0 \text{ and } \lim_{(r,x) \ra (r_0,x_0)} \sup_{y \in \R^n} \left| c(r,x,y) -c(r_0,x_0,y) \right| = 0
\end{equation*}
\end{enumerate}
Throughout this paper, we always assume the conditions I above.

\subsection{Main Theorems}
Our first theorem concerns about the existence of compatibility constant for Neumann Problem.
\begin{thm}[Compatibility Condition]\label{thm-main-c}
Consider the equation defined as follow:
\begin{equation} \label{eqn-main-com} \begin{cases}
a_{ij}(y) D_{ij} v(y) = f(y) &\text{ in } \R^n \setminus T^{\mathfrak a} \\
b^i(y) \left( \xi^i + D_iv(y) \right) + \gamma = g(y) &\text{ on } \partial T^{\mathfrak a}.
\end{cases} \end{equation}
Assume that
\begin{equation*}
\|a_{ij} \|_{C^\alpha(\R^n \setminus T^{\mathfrak a})} + \| b^i(\mathfrak{a}~\cdot) \|_{C^{1,\alpha}(\frac{1}{a}(\R^n \setminus T^{\mathfrak a}) )} \le \Lambda
\end{equation*}
and $\| f \|_{C^\alpha(\R^n \setminus T^{\mathfrak a})} + \| g \|_{C^{1,\alpha}(\R^n \setminus T^{\mathfrak a})}$ is bounded. Then, for any given $\xi \in \R^n$, there is a unique constant $\gamma = \gamma(\xi; (a_{ij}),b^i, f, g, \mathfrak{a})$ that makes the equation \eqref{eqn-main-com} has a soution $v$.

\end{thm}
\begin{definition} \item \begin{enumerate}[(i)]
\item
We are going to call $\gamma(\xi; (a_{ij}),b^i, f, g, \mathfrak{a})$ a  compatibility constant of the equation \eqref{eqn-main-com}.
\item
Now suppose that $f=g=0$. If $\gamma=\gamma((a_{ij}),b^i) = 0$ for all $\xi \in \R^n$ and the size of halls $\mathfrak{a}$,  then we call $(a_{ij})$ and $b^i$ (or the equation \eqref{eqn-main-1}) satisfies the compatibility condition.
\end{enumerate}\end{definition}

We remark that Laplace equation equipped the Neumann boundary condition satisfies the compatibility condition. We will show it in chapter \ref{sec-com}.

Now let us introduce our main theorem:
\begin{thm}[Main Theorem]\label{thm-main-1}
Let $u_{\e}$ be a viscosity solution of \eqref{eqn-main-1}. Suppose that our equation satisfies the conditions I and
\begin{enumerate}
\item
the equation \eqref{eqn-main-1} satisfies the compatibility condition,
\item
$u_\e$ is bounded uniformly on $\e$, and $u^* = u_*$ on the $\partial \Omega$ where $u^*$ and $u_*$ is same in definition \ref{def-lim-u}.
\item
$0<\mathfrak{a} \le {\mathfrak a}_0$ for uniform constant ${\mathfrak a}_0$ in theorem \ref{thm-cor2-elliptic}.
\end{enumerate}
Then, there exists an uniformly elliptic operator $\overline{L}$. And $u_\e$, solution of \eqref{eqn-main-1}, converges to $u$, solution of the equation \eqref{eqn-rst-1}, uniformly.
\begin{equation} \begin{cases}\label{eqn-rst-1}
\overline{L}(D^2 u, u, x) = 0 &\text{ in }\Omega \\
u = \vp(x) &\text{ on } \partial \Omega.
\end{cases} \end{equation}
\end{thm}

We will use the condition (3) to prove the uniformly ellipticity of $\overline{L}$. Hence it can be dropped if $a_{ij} = I$(Laplace case) and $b^i(y) = \nu^i$(Neumann boundary case). See chapter \ref{sec-cor2 and ellipticity}. And the condition (2) can be dropped if we can construct a barrier at any boundary points. We will show that such a barrier exists if $\Omega$ satisfies exterior sphere condition in chapter \ref{sec-Homogenization}.
\begin{cor} \label{cor-main-1}
Let $u_\e$ be the solution of the equation \eqref{eqn-main-1}. Assume that \eqref{eqn-main-1} satisfies all of the conditions (1) -(6), condition (3) in theorem \ref{thm-main-1} and the compatibility condition. Assume also that $\Omega$ satisfies an exterior sphere condition. then $u_\e$ converges uniformly to $u_0$ which is the solution of \eqref{eqn-rst-1}.
\end{cor}

\begin{cor} \label{cor-main-2}
Let $u_\e$ be the solution of the equation \eqref{eqn-main-0} with the condition \eqref{con-conti}. Then $u_\e$ converges uniformly to $u_0$ which is the solution of \eqref{eqn-rst-0} if the domain $\Omega$ satisfies an exterior sphere condition.
\end{cor}

Finally, we develop the following estimate. It tells us that $u_\e$ is almost Lipschitz continuous.

\section{Existence and Regularity}\label{sec-exist}
\subsection{}
We begin by recalling the definition of viscosity solutions. It is a definition of viscosity solution defined in \cite{CIL}.
\begin{definition}
For a given function $u$ defined on $\overline{\Omega}$,
\begin{enumerate}
\item
 the superdifferential $D^{2,+}u(x)$ of order 2 at $x \in \overline{\Omega}$ is defined by
\begin{equation*} \begin{aligned}
D^{2,+}u(x) = \{(p,M) \in \R^n \times \mathbb{S}^n : u(x+h) \le u(x) + <p,h> \qquad\\
 + \frac{1}{2}\langle Mh,h\rangle+ o(|h|^2) \text{ as } x+h \in \overline{\Omega} \text{ and } h \ra 0  \}
\end{aligned} \end{equation*}
\item
 the subdifferential $D^{2,-}u(x)$ of order 2 at $x \in \overline{\Omega}$ is defined by
\begin{equation*} \begin{aligned}
D^{2,-}u(x) = \{(p,M) \in \R^n \times \mathbb{S}^n : u(x+h) \ge u(x) + <p,h> \qquad\\
 + \frac{1}{2}<Mh,h> + o(|h|^2) \text{ as } x+h \in \overline{\Omega} \text{ and } h \ra 0  \}
\end{aligned} \end{equation*}
\item
$\overline{D}^{2,+}u(x)$ is the set of those points $(r,p,M) \in \R \times \R^n \times \mathbb{S}^n$ to which there corresponds a sequence $\{(x_n,p_n,M_n)\} \in \overline{\Omega} \times \R^n \times \mathbb{S}^n$ such that $(p_n,M_n) \in D^{2,+}u(x_n)$ for $n \in \mathbb{N}$ and such that $x_n \ra x$, $u(x_n) \ra r$, $p_n \ra p$, and $M_n \ra M$ as $n \ra \infty$.
\item
$\overline{D}^{2,-}u(x)$ is the set of those points $(r,p,M) \in \R \times \R^n \times \mathbb{S}^n$ to which there corresponds a sequence $\{(x_n,p_n,M_n)\} \in \overline{\Omega} \times \R^n \times \mathbb{S}^n$ such that $(p_n,M_n) \in D^{2,-}u(x_n)$ for $n \in \mathbb{N}$ and such that $x_n \ra x$, $u(x_n) \ra r$, $p_n \ra p$, and $M_n \ra M$ as $n \ra \infty$.
\end{enumerate}
\end{definition}

\begin{definition} \label{def-viscositysol}
For a given nonlinear equation that is defined in bounded domain $\Omega$,
\begin{equation}\label{eqn-nonlinear} \begin{cases}
F(D^2u,Du,u,x) = f(x) &\text{ in } \Omega \\
G(Du,u,x) = g(x) &\text{ on } \partial \Omega
\end{cases} \end{equation}
\begin{enumerate}
\item
an upper semi-continuous $u$ is a viscosity sub-solution of \eqref{eqn-nonlinear} if
\begin{equation*} \begin{cases}
F(M,p,u(x),x) \ge 0 \qquad \text{ for } x \in \Omega, (p,M) \in \overline{D}^{2,+}u(x) \\
F(M,p,u(x),x) \ge 0 \text{ or } G(p,u(x),x) \le 0 \quad \text{ for } x \in \partial \Omega, (p,M) \in \overline{D}^{2,+}u(x). \\
\end{cases} \end{equation*}

\item
an lower semi-continuous $u$ is a viscosity super-solution of \eqref{eqn-nonlinear} if
\begin{equation*} \begin{cases}
F(M,p,u(x),x) \le 0 \qquad \text{ for } x \in \Omega, (p,M) \in \overline{D}^{2,-}u(x) \\
F(M,p,u(x),x) \le 0 \text{ or } G(p,u(x),x) \ge 0 \quad \text{ for } x \in \partial \Omega, (p,M) \in \overline{D}^{2,-}u(x) \\
\end{cases} \end{equation*}

\item
$u$ is called a viscosity solution of \eqref{eqn-nonlinear} if $u$ is called a viscosity super- and sub-solution of \eqref{eqn-nonlinear}.
\end{enumerate} \end{definition}

We employ the comparison principle given at \cite{CIL}.
\begin{lem}[Comparison Principle] \item
Let $v^+_e$ and $v^-_\e$  be  viscosity super- and sub-solutions of \eqref{eqn-main-1} respectively for given $\e <1$. Then $v^+_\e(x)\geq v^-_\e (x)$ in $\Omega$.
\end{lem}

The following existence theorem for the viscosity solution can be found at  \cite{CIL}.
\begin{lem}[Existence]
\item There is a unique viscosity solution $u_{\e}$ of \eqref{eqn-main-1}.
\end{lem}

We finish this section by introducing more intuitive concept of viscosity solution that is equivalent the definition above if our solution $u$ is in $\mathcal{C}^1$ near the boundary:

\begin{definition}
Let $x \in \Omega$ for some bounded $\Omega$. Then a continuous function $u$ is the viscosity super(sub)-solution of the equation \eqref{eqn-nonlinear} at $x$ if $\phi \in \mathcal{C}^2(\Omega)$ touches $u$ by below at $x$, then
\begin{equation*}
F(D^2\phi(x),D\phi(x),\phi(x),x) \le(\ge) f(x).
\end{equation*}
\end{definition}

\begin{lem}
Suppose that $u$ is a continuous viscosity super-solution(sub-solution) of \eqref{eqn-nonlinear} at all interior point of $\Omega$. Suppose also that $u$ is $\mathcal{C}^1$ near the $\partial \Omega$. Then $u$ is a viscosity super-solution(sub-solution) of \eqref{eqn-nonlinear} if
\begin{equation*}
G(Du,u,x) \ge(\le) 0.
\end{equation*}
\end{lem}

\section{Compatibility Condition}\label{sec-com}
In this section, we are going to define the compatibility condition and investigate their properties.
\subsection{Existence and Regularity of Periodic Viscosity Solution}
Before introducing the compatibility condition, we are going to find the (periodic) viscosity solution of  the following equation defined on $\R^n \setminus T^{\mathfrak a}$:
\begin{equation} \label{eqn-precom} \begin{cases}
a_{ij}(y) D_{ij} v_\e(y) = f(y) &\text{ in } \R^n \setminus T^{\mathfrak a} \\
b^i(y) D_i v_\e(y) + \e^2 v_\e(y) = g(y) &\text{ on } \partial T^{\mathfrak a}.
\end{cases} \end{equation}
We assume that all the functions $a_{ij}$, $f$, $b^i$ and $g$ are periodic in $y$ variable. Assume also that $a_{ij}$ is uniformly elliptic with elliptic constant $\lambda$ and $\Lambda$ and $b^i(y) \nu^i \ge \mu$  for $\mu > 0$ in condition \eqref{con-oblique}.
Then, we will prove the comparison for the viscosity solution of \eqref{eqn-precom} .
\begin{lem}[Comparison Principle] \label{lem-precom-cp}
Let $v^+$ and $v^-$ be continuous, bounded and periodic viscosity super and sub-solution of the
equation \eqref{eqn-precom} respectively. Then we have
$$ v^+ \ge v^- .$$
\end{lem}
\begin{proof}
First, assume that $v^+$ and $v^-$ are in $\mathcal{C}^2$.  If the conclusion is not true,  there exists $y_0 \in \R^n \setminus T^{\mathfrak a} $ such that
$v^+(y_0) < v^-(y_0)$. Now we add a positive constant $c>0$ so that
$v^+(y)+c > v^-(y)$ and then decrease  $ c$ until $v^+(y)+c $ touches $v^-(y)$.
Set $c_1=\min \{c>0:v^+(y)+c \ge v^-(y)\quad\text{ for all $y$}\}$. Then, from the assumption, $c_1>0$ and hence we can find $y_1 \in \overline{\R^n \setminus T^{\mathfrak a}}$ such that $v^+(y)+c_1$ touches $v^-(y)$ at $y_1$ from above .
In other words,
\begin{equation*}\begin{aligned}
v^+(y) + c_1 &\ge v^-(y) \quad\text{ for all $y\in \R^n \setminus T^{\mathfrak a} $}\\
v^+(y_1) + c_1 &= v^-(y_1).
\end{aligned}\end{equation*}

First let us consider the case when $y_1 \in \partial T^{\mathfrak a}$. Then, from $v^-(y_1)=v^+(y_1)+c_1$ and $v^-(y) \ge v^+(y)+c_1$, we have
\begin{equation*} \begin{aligned}
0&\ge b^i(y_1) D_iv^-(y_1) + \e^2 v^-(y_1) \\
&\ge b^i(y_1) D_iv^+(y_1) + \e^2 v^+(y_1) + \e^2 c_1 \\
&= \e^2 c_1 > 0,
\end{aligned} \end{equation*}
which  is a contradiction. Hence $y_1$ is not on the boundary.
Therefore  $y_1$ is supposed to be an interior point of $\R^n \setminus T^{\mathfrak a} $.

Hence, $y_0$ the touching point, should be in the interior of $\R^n \setminus T^{\mathfrak a}$. But, it also impossible because $v^+ + c_0$ also be a viscosity super solution of \eqref{eqn-precom} and the super solution $v^+ + c_0$ cannot touch the sub-solution $v^-$ by above at any interior point. So, $v^+ \ge v^-$ on $\R^n \setminus T^{\mathfrak a}$.

The case $v^+$ is lower semicontinuous and $v^-$ is upper-semicontinuous can be proved by the usual viscosity argument. See chapter 3 in \cite{CC}.
\end{proof}

From the comparison, we directly prove the existence of the solution of equation \eqref{eqn-precom}.
\begin{lem} \label{lem-precom-ex}
Suppose that there exist a periodic bounded continuous (viscosity) super-solution $h^+$ and a sub-solution $h^-$ of the equation of equation \eqref{eqn-precom}. Then, there exists the unique periodic viscosity solution $v_\e$ of \eqref{eqn-precom} located between $h^+$ and $h^-$.
\end{lem}
\begin{proof}
We first define $v = \inf\{h : h \text{ is periodic, bounded viscosity super-solutions }  \}$. Then, $v$ is well defined because of $h^+$ and $h^-$. And, we also prove that $v$ is a viscosity solution  by applying the argument in \cite{CIL}. Finally, from the definition,
\begin{equation*} \begin{aligned}
v(y) &= \inf \{ h(y) : h \text{ is peroicid, bounded viscosity solution of \eqref{eqn-precom} } \} \\
&= \inf \{ h(y+m) : h \text{ is peroicid, bounded viscosity solution of \eqref{eqn-precom} } \} \\
&= v(y+m)
\end{aligned} \end{equation*}
for all $m \in \Z^n$.
\end{proof}

\begin{lem}\label{lem-precom-bdd}
For each $\e>0$, there exists the solution $v_\e$ of the equation \eqref{eqn-precom} satisfying
\begin{equation*}
\| \e^2 v_\e \|_{L^\infty(\R^n \setminus T^{\mathfrak a})} \le C(\|f\|_{L^\infty(\R^n \setminus T^{\mathfrak a})} + \|g \|_{L^\infty(\R^n \setminus T^{\mathfrak a})})
\end{equation*}
for some constant $C=C(n,\Lambda,\lambda,\mu,\mathfrak{a})$. 
\end{lem}
\begin{proof}
For given $f$ and $g$, let $v^1$ be the solution of the equation \eqref{eqn-precom} when $g=0$ and $v^2$ be the solution of \eqref{eqn-precom} when $f=0$. Then, if we have $\|\e^2 v^1\| \le C(n,\Lambda,\lambda,\mu,\mathfrak{a})|f|_\infty $ and $\|\e^2 v^2\| \le C(n,\Lambda,\lambda,\mu,\mathfrak{a})|g|_\infty$, the conclusion comes from the linearity of the equation. So, we consider the case $g=0$ first. We may assume that $|f|_\infty = 1$ without losing generality.

Select a ball $B = B_\mathfrak{a}(0)$ which is a component of $T^{\mathfrak a}$ and a unit cube $Q$ of $\R^n$ with center 0. Set
\begin{equation*}
h=-\displaystyle\frac{1}{|x|^\alpha} + \frac{1}{\mathfrak{a}^\alpha} \ge 0.
\end{equation*}
By the direct calculation, we can obtain
\begin{equation*}
a_{ij}(y) D_{ij} h(y) \le \cM^+(D^2h(y)) \le -\alpha \displaystyle\frac{1}{|x|^{\alpha+2}}(\Lambda(\alpha-1) + \lambda(n-1)) \text{ in } \R^n \setminus B_\mathfrak{a}
\end{equation*}
where $\cM^+(M) = \sum_{e_i > 0} \Lambda e_i + \sum_{e_i \le 0} \lambda e_i$ and $e_i$ are eigenvalues of $M$. 

Select a large  $\alpha=\alpha(n,\lambda, \Lambda)$ so  that $\cM^+(D^2h(y)) \le 0$. and define 
$$\widetilde{h}(y) = \displaystyle\frac{1}{\beta} h(y).$$
where $\beta(\alpha) = -\alpha \displaystyle\frac{1}{\sqrt{n}^{\alpha+2}}(\Lambda(\alpha-1) + \lambda(n-1))$.

Then, we have
\begin{equation*}
a_{ij}(y) D_{ij} \widetilde{h}(y) \le -1 \text{ in } Q \setminus B_\mathfrak{a}.
\end{equation*}
From $D\widetilde{h}(y) = -\displaystyle\frac{\alpha}{\beta a^{\alpha +1}} \cdot \nu$ on the boundary $\partial  B_\mathfrak{a}$, we get
\begin{equation*}
b^i(y) D_i \widetilde{h} = b^i(y) \displaystyle\frac{\alpha}{\beta a^{\alpha +1}} \nu_i \ge  - \displaystyle\frac{\alpha}{\beta a^{\alpha +1}}.
\end{equation*}
Now we define
\begin{equation*}
\delta =  \displaystyle\frac{\alpha}{\beta a^{\alpha +1}} \text{ and,}
\end{equation*}
\begin{equation*}
v_+ = \min_{m \in \Z^n} ( \widetilde{h}(y-m) ) + \frac{1}{\e^2} \delta.
\end{equation*}
Then, because of the shape of $\widetilde{h}$, $a_{ij}(y) D_{ij} v^+(y) \le -1$ in the viscosity sense for all interior points of $\R^n\setminus T^{\mathfrak a}$ and, on the boundary, $G(Dv_+, y) + \e^2 v_+ \ge -\delta + \delta =0$ on $\partial T^{\mathfrak a}$.
Therefore $v_+$ is a periodic viscosity super-solution, and we can observe that $-v^+$ be a viscosity sub-solution. Hence, by the lemma \ref{lem-precom-ex}, there is a  periodic viscosity solution $v_{\e}$ of the equation \eqref{eqn-precom} between a sub-solution $-v^+$ and a super-solution $v_+>0$. In addition, we have
$$-\left( \e^2 \frac{1}{a^\alpha} + \delta \right) \le -\e^2 v_+ \le \e^2 v_\e \le \e^2 v_+ \le \e^2 \frac{1}{a^\alpha} + \delta $$
and
$$\| \e^2 v_\e \|_{L^\infty(\R^n \setminus T^{\mathfrak a})}\le C(n,\lambda, \Lambda, \mu, \mathfrak{a}).$$
For general $f$, consider the function $\displaystyle\frac{v}{\| f \|_{L^\infty(\R^n \setminus T^{\mathfrak a})}}$ and apply the above estimate, we can get the conclusion. And, if $f=0$, then $\e^2 v_\e = \pm \| g \|_{L^\infty(\R^n \setminus T^{\mathfrak a})}$ become a super and sub-solutions so, we can deduce that $\| \e^2 v_\e^2 \|_{L^\infty(\R^n \setminus T^{\mathfrak a})} \le \| g \|_{L^\infty(\R^n \setminus T^{\mathfrak a})}$.
\end{proof}

\begin{lem} \label{lem-precom-osc}
Let $v_{\e}$ be the solution of \eqref{eqn-precom} and
\begin{equation*}
\hat{v}_\e(y) = v_\e(y) - \min_{y \in \R^n \setminus T^{\mathfrak a}} v_\e(y).
\end{equation*}
Then we have
\begin{equation*}
\osc_{\R^n \setminus T^{\mathfrak a}} v_\e = \osc_{\R^n \setminus T^{\mathfrak a}} \hat{ v}_\e\leq C \left( \| f \|_{L^\infty(\R^n \setminus T^{\mathfrak a})} + \| g \|_{L^\infty(\R^n \setminus T^{\mathfrak a})} \right)
\end{equation*}
for a uniform constant $C=C(n,\lambda, \Lambda, \mu, \mathfrak{a})$.
\end{lem}
\begin{proof}
For small $\delta>0$, $T^{\mathfrak{a}+\delta}$ be a set of balls with radius $\mathfrak{a}+\delta$ instead of balls with radius $\mathfrak{a}$. Let $S_* = \sup_{\R^n \setminus T^{\mathfrak{a}+\delta}} \hat{v}_\e$, $I_* = \inf_{\R^n \setminus T^{\mathfrak{a}+\delta}} \hat{v}_\e$, $S_0 = \sup_{\R^n \setminus T^{\mathfrak a}} \hat{v}_\e = \osc \hat{v}_\e$, $I_0=\inf_{\R^n \setminus T^{\mathfrak a}} \hat{v}_\e=0$ and $\gamma_\e = \min_{y \in \R^n \setminus T^{\mathfrak a}} \e^2 v_\e(y)$.

Then  $\hat{v}_\e$ satisfies
\begin{equation} \label{eqn-precom-hat} \begin{cases}
a_{ij}(y) D_{ij} \widehat{v}_\e(y) = f(y) &\text{ in } \R^n \setminus T^{\mathfrak a} \\
b^i(y) D_i\widehat{v}_\e + \e^2 \widehat{v}_\e(y) = g(y) - \gamma_\e &\text{ on } \partial T^{\mathfrak a}.
\end{cases} \end{equation}
Let  $Q$ be a unit cell of  $\R^n \setminus T^{\mathfrak a}$ which is punctured by a ball $B_\mathfrak{a}(y_0)$. We may assume that $y_0 = 0$. Since $\hat{v}_\e$ is nonnegative and $\overline{Q\setminus T^{a+\delta}}$ is contained in $\R^n \setminus T^{\mathfrak a}$, we can apply the Harnack estimate (in \cite{CC}) on $\widehat{v}_\e$ in $Q \setminus T^{\mathfrak{a}+\delta}$, and hence we have
\begin{equation*}
\sup_{Q \setminus T^{\mathfrak{a}+\delta}} \widehat{v}_\e \le C_1 \left( \inf_{Q \setminus T^{\mathfrak{a}+\delta}} \widehat{v}_\e + \| f \|_{L^\infty(\R^n \setminus T^{\mathfrak a})} \right)
\end{equation*}
for some $C_1$ which depend only on $n$, $\lambda$, $\Lambda$, $\mathfrak{a}$ and $\delta$.
And hence, from the periodicity, we have
\begin{equation*}
S_* \le C_1 \left( I_* + \| f \|_{L^\infty(\R^n \setminus T^{\mathfrak a})} \right).
\end{equation*}

Let $v^+ = -\displaystyle\frac{K}{2}(|y|^2 - (\mathfrak{a}+\delta)^2 ) + S_* $ for some $K>0$.
Then we have
\begin{equation*}\begin{aligned}
&a_{ij}(y) D_{ij} v^+ \le - n \lambda K &&\text{ in } B_{\mathfrak{a}+\delta} \setminus T^{\mathfrak a} \\
&v^+ \ge S_* \ge \hat{v}_\e &&\text{ on } \partial B_{\mathfrak{a}+\delta} \\
\end{aligned}\end{equation*}
and
\begin{equation*} \begin{aligned}
b^i(y) D_i v^+ + \e^2 v^+ &= K a y \cdot \nu  + \e^2 v^+(y) \\
&\ge K a \mu  + \e^2 v^+(y) \\
&\ge  K a \mu
\end{aligned} \end{equation*}
on $\partial B_\mathfrak{a}$.

Then, if we select $K = \displaystyle\frac{1}{n\lambda} \| f \|_{L^\infty(\R^n \setminus T^{\mathfrak a})} + \frac{1}{a \mu}\left( \| g \|_{L^\infty(\R^n \setminus T^{\mathfrak a})} + | \gamma_\e | \right)$, $v^+$ is a super solution in $B_{\mathfrak{a}+\delta} \setminus B_\mathfrak{a}$. So, a comparison principle tells us  $v^+ \ge  \hat{v}_\e$ in $B_{\mathfrak{a}+\delta} \setminus T^{\mathfrak a}$ and hence
\begin{equation*}
S_0 \le - \displaystyle\frac{K}{2}(\mathfrak{a}^2 - (\mathfrak{a}+\delta)^2 ) + S_*.
\end{equation*}
So, by choosing $\delta$ properly between $0$ and $\frac{1}{2} - \mathfrak{a}$, we have $S_0 \le S_* + C_2 K$ for some constant $C_2(\mathfrak{a}) > 0$.
Similarly, We can obtain $I_0 \ge I_* -C_2 K$.

Combine these three results. Then we can conclude
\begin{equation*} \begin{aligned}
S_0 &\le S_* + C_2 K \\
&\le C_1 (I_* + \| f \|_{L^\infty(\R^n \setminus T^{\mathfrak a})}) + C_2 K \\
&\le C_1 ( I_0 + C_2 K )  + C_1 \| f \|_{L^\infty(\R^n \setminus T^{\mathfrak a})} + C_2 K \\
&\le C \left( \| f \|_{L^\infty(\R^n \setminus T^{\mathfrak a})} + \| g \|_{L^\infty(\R^n \setminus T^{\mathfrak a})} + | \gamma_\e | \right)
\end{aligned} \end{equation*}
where $C$ depends only on $n, \lambda, \Lambda, \mu$ and $\mathfrak{a}$.
Finally, applying lemma \ref{lem-precom-bdd} to get $|\gamma_\e| \le   C \left( \| f \|_{L^\infty(\R^n \setminus T^{\mathfrak a})} + \| g \|_{L^\infty(\R^n \setminus T^{\mathfrak a})} \right)$, we get the conclusion.
\end{proof}

Now, we are discussing the regularity of $v_\e$. The regularity of viscosity solution of bounded domain has been developed by  many authors. Especially, we will use the results in \cite{CC}, \cite{LT} and \cite{GT} to get the regularity of $v_\e$. Let us assume that $\|a_{ij} \|_{C^\alpha(\R^n \setminus T^{\mathfrak a})} + \| b^i(\mathfrak{a}~\cdot) \|_{\frac{1}{\mathfrak{a}} \left( C^{1,\alpha}(\R^n \setminus T^{\mathfrak a}) \right)} \le \Lambda$ and $\| f \|_{C^\alpha(\R^n \setminus T^{\mathfrak a})} + \| g \|_{C^{1,\alpha}(\R^n \setminus T^{\mathfrak a})}$ is bounded.
Let $Q$ be a cell of $\R^n$. We may assume that the center of $Q$ and $B_\mathfrak{a}$ is $0$. By applying the interior estimate in \cite{CC}, $\widehat{v}_\e$ is $C^2$ at every interior points and hence for some open set $\widetilde{Q}$ which is contained in $\R^n \setminus \overline{T^{\mathfrak a}}$ and containing $\partial Q$, and
\begin{equation*} \begin{aligned}
\| \widehat{v}_\e \|_{C^{2,\alpha}(\widetilde{Q})} &\le C \left( \| \widehat{v}_\e \|_{L^\infty(Q\setminus B_a)} + \| f \|_{C^\alpha(Q\setminus B_a
)} \right) \\
&\le C \left( \| f \|_{C^\alpha(\R^n \setminus T^{\mathfrak a})} + \| g \|_{C^{1,\alpha}(\R^n \setminus T^{\mathfrak a})} \right)
\end{aligned} \end{equation*}
where $C$ is a constant depending only on $n$, $\lambda$, $\Lambda$ and $\mathfrak{a}$. Let $\phi(y)$ be a function which has same value with $\widehat{v}_\e$ in $\widetilde{Q}$. Then, $\widehat{v}_\e$ satisfies
\begin{equation*} \begin{cases}
a_{ij}(y) D_{ij} \widehat{v}_\e(y) = f(y) &\text{ in } Q \setminus B_\mathfrak{a} \\
b^i(y) D_i\widehat{v}_\e + \e^2 \widehat{v}_\e(y) = g(y) - \gamma_\e &\text{ on } \partial B_\mathfrak{a} \\
\widehat{v} = \phi &\text{ on } \partial Q.
\end{cases} \end{equation*}

So, from the \cite{LT}, $\widehat{v}_\e$ is $\mathcal{C}^{2,\alpha}$ in $\overline{Q \setminus B_\mathfrak{a}}$ (hence in $\overline{\R^n \setminus T^{\mathfrak a}}$) and, from the a priori estimate in \cite{GT}, we have the $\mathcal{C}^{2,\alpha}$ estimate
\begin{equation} \begin{aligned} \label{eqn-precom-c2}
\| \widehat{v}_\e \|_{C^{2,\alpha}(\R^n \setminus T^{\mathfrak a})} &\le C \left( \| \widehat{v}_\e \|_{L^\infty(\R^n \setminus T^{\mathfrak a})} + \| \phi \|_{C^{2,\alpha}(\R^n \setminus T^{\mathfrak a})} + \| f \|_{C^\alpha(\R^n \setminus T^{\mathfrak a})} + \| g \|_{C^{1,\alpha}(\R^n \setminus T^{\mathfrak a})} \right) \\
&\le C \left( \| \widehat{v}_\e \|_{L^\infty(\R^n \setminus T^{\mathfrak a})} + \| f \|_{C^\alpha(\R^n \setminus T^{\mathfrak a})} + \| g \|_{C^{1,\alpha}(\R^n \setminus T^{\mathfrak a})} \right) \\
&\le C \left( \| f \|_{C^\alpha(\R^n \setminus T^{\mathfrak a})} + \| g \|_{C^{1,\alpha}(\R^n \setminus T^{\mathfrak a})} \right)
\end{aligned} \end{equation}

In summary, we have
\begin{lem} \label{lem-precom-c2}
Let $v$ is the (viscosity) solution of the equation \eqref{eqn-precom} with $\| a_{ij} \|_{C^\alpha(\R^n \setminus T^{\mathfrak a})} + \| b^i(\mathfrak{a}~\cdot) \|_{ C^{1,\alpha}(\frac{1}{\mathfrak{a}}(\R^n \setminus T^{\mathfrak a}))} \le \Lambda$ and  $\| f \|_{C^{\alpha}(\R^n \setminus T^{\mathfrak a})}+\| g \|_{C^{1,\alpha}(\R^n \setminus T^{\mathfrak a})}<\infty$. Then $v$ is in $\mathcal{C}^{2,\alpha}(\R^n \setminus T^{\mathfrak a})$ for some $\alpha$ and we have
\begin{equation}
\| \widehat{v}_\e \|_{C^{2,\alpha}(\R^n \setminus T^{\mathfrak a})} \le C \left( \| f \|_{C^\alpha(\R^n \setminus T^{\mathfrak a})} + \| g \|_{C^{1,\alpha}(\R^n \setminus T^{\mathfrak a})} \right).
\end{equation}
\end{lem}

\subsection{The existence and uniqueness of $\gamma$}
In this section, we are going to prove the theorem \ref{thm-main-1} by applying previous subsection.
First, for fixed $\xi \in \R^n$, consider the following approximated equation
\begin{equation} \label{eqn-precor1-e} \begin{cases}
a_{ij}(y) D_{ij}v_\e = f(y) &\text{ in } \R^n \setminus T^{\mathfrak a} \\
b^i(y) \left( \xi^i + D_i v_\e \right) + \e^2 v_\e = g( y) &\text{ on } \partial T^{\mathfrak a}.
\end{cases} \end{equation}

For each $\e$, we have the periodic viscosity solution $v_\e = v_\e(y;\xi)$ of \eqref{eqn-precor1-e} by lemma \ref{lem-precom-bdd}.
\begin{lem} \label{lem-precor1-ex}
For each $\e>0$, there exists $v_\e$ satisfying \eqref{eqn-precor1-e} and we have
\begin{equation*}
\| \e^2 v_\e \|_{L^\infty(\R^n \setminus T^{\mathfrak a})} \le C(n,\Lambda,\lambda,\mu,\mathfrak{a}) \left( \| f \|_{L^\infty(\R^n \setminus T^{\mathfrak a})} + \| g \|_{L^\infty(\R^n \setminus T^{\mathfrak a})} + |\xi| \right).
\end{equation*}
\end{lem}

Actually, we just need the result when $f=g=0$. But, in this section, we can consider more general case($f$and $g$ are not identically $0$) because it does not effect on the result. The following is about the oscillation of $\widehat{v}_\e$.

\begin{lem} \label{lem-precor1-osc}
Let $v_{\e}$ be the solution of the equation \eqref{eqn-precor1-e} and
\begin{equation*}
\hat{v}_\e(y) = v_\e(y) - \min_{y \in \R^n \setminus T^{\mathfrak a}} v_\e(y).
\end{equation*}
Then we have
\begin{equation*}
\osc_{\R^n \setminus T_1} \hat{ v}_\e\leq C(n,\Lambda,\lambda,\mu,\mathfrak{a}) \left( \| f(\cdot) \|_{L^\infty(\R^n \setminus T^{\mathfrak a})} + \| g(\cdot) \|_{L^\infty(\R^n \setminus T^{\mathfrak a})} + |\xi| \right).
\end{equation*}
\end{lem}
\begin{proof}
Let $\gamma_\e = \min_{y \in \R^n \setminus T^{\mathfrak a}} \e^2  v_\e(y)$.

Then  $\hat{v}_\e$ satisfies
\begin{equation} \label{eqn-precor1-hat} \begin{cases}
a_{ij}(y) D_{ij} \widehat{v}_\e(y) = f(y) &\text{ in } \R^n \setminus T^{\mathfrak a} \\
b^i(y) D_i\widehat{v}_\e + \e^2 \widehat{v}_\e(y) = g(y) - \gamma_\e - b^i(y) \xi^i &\text{ on } \partial T^{\mathfrak a}.
\end{cases} \end{equation}
Now apply lemma \ref{lem-precom-osc} and then we can get the conclusion.
\end{proof}

We can also obtain the estimate of $\widehat{v}_\e$ by lemma \ref{lem-precor1-osc} and lemma \ref{lem-precom-c2}.
\begin{lem} \label{lem-com-e-c2}
Suppose that $a_{ij}$, $f$, $b^i$ and $g$ satisfies the condition in lemma \ref{lem-precom-c2}. Then, we have
\begin{equation*}
\| \widehat{v}_\e \|_{C^{2,\alpha}(\R^n \setminus T^{\mathfrak a})} \le C \left( \| f \|_{C^\alpha(\R^n \setminus T^{\mathfrak a})} + \| g \|_{C^{1,\alpha}(\R^n \setminus T^{\mathfrak a})} + |\xi| \right)
\end{equation*}
where $C$ is depending only on $n$, $\lambda$, $\Lambda$ and $\mathfrak{a}$.
\end{lem}
\begin{proof}
\begin{equation*} \begin{aligned}
&\| b^i(\cdot) \|_{C^{1,\alpha}(\R^n \setminus T^{\mathfrak a})} \le \displaystyle\frac{1}{a^{1+\alpha}} \Lambda \\
&|\gamma_\e| \le C(n,\Lambda,\lambda,\mu,\mathfrak{a}) \left( \| f \|_{L^\infty(\R^n \setminus T^{\mathfrak a})} + \| g \|_{L^\infty(\R^n \setminus T^{\mathfrak a})} + |\xi| \right)
\end{aligned} \end{equation*}
from the condition \eqref{con-alpha} in chapter 1 and lemma \ref{lem-precor1-osc}, $\| g(\cdot) -\gamma_\e -b^i(\cdot) \xi^i \|_{C^{1,\alpha}(\R^n \setminus T^{\mathfrak a})}$ is bounded. Hence we can apply lemma \ref{lem-precom-c2} to $\widehat{v}_\e$ of \eqref{eqn-precor1-hat} and we can get the estimate.
\end{proof}
Lemma \ref{lem-com-e-c2} and \ref{lem-precor1-ex} tells us that $\| \widehat{v}_\e \|_{C^{2,\alpha}} + \| \e^2 v_\e \|_{L^{\infty}}$ is bounded uniformly on $\e$.
So, from  Arzela-Ascoli theorem, we can deduce that there is a $v \in \mathcal{C}^2(\R^n \setminus T^{\mathfrak a})$, $\gamma \in \R$, and a subsequence $ \{ \e_j \}$ where $\widehat{v}_{\e_j}$ converges $v$ in $\mathcal{C}^2(\R^n \setminus T^{\mathfrak a})$ and $\e^2 v_\e \ra \gamma$ uniformly.

And if we take $j \ra \infty$, then $v \in C^{2,\alpha}$ and $\alpha$ satisfy the equation \eqref{eqn-main-com}.
\begin{prop} \label{prop-exist-alpha}
If $a_{ij}$, $f$, $b^i$ and $g$ satisfies the condition in lemma \ref{lem-precom-c2} then, we always find $\gamma = \gamma(\xi;(a_{ij}),(b^i),f,g,\mathfrak{a})$ and $v = v(y;\xi) \in \mathcal{C}^{2,\alpha}$ which satisfy the equation \eqref{eqn-main-com}.
\end{prop}

\begin{lem}[Uniqueness of $\gamma$] \label{lem-alpha-uniq}
Let $\gamma$ be given as  \ref{prop-exist-alpha}. Then such $\gamma$ is unique.
\end{lem}
\begin{proof}
Let $v^1(y)$ and $v^2(y)$ be two solutions of the equation \eqref{eqn-main-com} with corresponding to constants $\gamma^1$ and $\gamma^2$ respectively. And, to obtain a contradiction, assume that $\gamma^1$ and $\gamma^2$ are not same. We may assume that $\gamma^1 < \gamma^2$ without losing generality. Since $v^1$ and $v^2$ are bounded, we can find a constant $c$  such that $v^1+c$ touches $v^2$ by above at $y_0 \in \R^n \setminus T^{\mathfrak a}$.
Suppose that $y_0$ is a interior point, then $(v^1+c)-v^2$ has a local minimum at $y_0$. but since $(v^1+c)-v^2$ is a solution of $a^{ij}(y) D_{ij} ((v^1+c)-v^2) = 0$, $(v^1+c)-v^2$ cannot have its minimum at interior point because of the strong maximum principle. So $y_0$ cannot be in the interior of $\R^n \setminus T^{\mathfrak a}$. Suppose that $y_0 \in \partial T^{\mathfrak a}$. Then, $G(\xi+D(v^1+c),y) + \gamma^2 \le G( \xi + Dv^2,y) + \gamma^2 \le g(y)$ but,
\begin{equation*}
G(\xi+D(v^1+c),y) + \gamma^2 =
G(\xi+Dv^1,y) + \gamma^2 = g(y) -\gamma^1 + \gamma^2 > g(y).
\end{equation*}
So we get a contradiction and hence $\gamma^1$ should be the same with $\gamma^2$.
\end{proof}

\begin{proof}[\bf{Proof of theorem \ref{thm-main-c}}]
From proposition \ref{prop-exist-alpha}, there exist $\gamma$ that makes the equation \ref{eqn-main-com} has a solution $v$. And by lemma \ref{lem-alpha-uniq}, such a $\gamma$ is unique. 
\end{proof}
\begin{remark}
We can define a compatibility constant even the operator and boundary condition are nonlinear. More precisely, For given operator $F(M,y)$ and boundary condition $G(p,y)$, and a vector $\xi \in \R^n$, there is a constant $\alpha$ and a periodic function $v(y) \in \mathcal{C}^2(\R^n \setminus T^{\mathfrak a})$ satisfying the equation
\begin{equation} \begin{cases}
F(D^2v, r, x_0, y) = f(x_0, y) &\text{ in } \R^n \setminus T^{\mathfrak a} \\
G(\xi + Dv, y) + \gamma = g(x_0, y) &\text{ on } \partial T^{\mathfrak a}.
\end{cases} \end{equation}
if the operator satisfies the conditions in \cite{LT}. The proof is quite similar.
\end{remark}

\subsection{Examples satisfying the compatibility condition}
As we told in the introduction, the Laplace equation and the Neumann boundary condition satisfies the compatibility condition. Let $Q$ be a one cell of $\R^n \setminus T^{\mathfrak a}$ having center $0$ and $v$ is a solution satisfying the following equation:
\begin{equation*} \begin{cases}
\triangle v = 0 &\text{ in } \R^n \setminus T^{\mathfrak a} \\
\displaystyle\frac{\partial v}{\partial \nu} + \gamma = 0 &\text{ on } \partial T^{\mathfrak a}.
\end{cases} \end{equation*}

Then, by using divergence theorem, we have 
\begin{equation*} \begin{aligned}
0&=\int_{Q \setminus T^{\mathfrak a}} \triangle v dx = \int_{\partial Q} \displaystyle \frac{\partial v}{\partial \nu} d\sigma_x + \int_{\partial (T^{\mathfrak a} \cap Q)} \frac{\partial v}{\partial \nu} d\sigma_x \\
&=0 + \int_{\partial (T^{\mathfrak a} \cap Q)} - \xi \cdot \nu -\gamma d\sigma_x \\
&=-\gamma \int_{\partial (T^{\mathfrak a} \cap Q)} d\sigma_x \\
&=-\gamma |\partial (T^{\mathfrak a} \cap Q)|.
\end{aligned} \end{equation*}
Hence, $\gamma$ should be $0$.

Moreover, the operator satisfying the symmetric condition $a_{ij}(-y) = a_{ij}(y)$ and $b^i(-y) = -b^i(y)$ Then, we can show that $(a_{ij})$ and $b^i$ satisfies the compatibility condition.
\begin{prop}
Let $v_\e$ be the solution of equation \eqref{eqn-precor1-e}. Assume that $a_{ij}$ and $b^i$ satisfies $a_{ij}(-y) = a_{ij}(y)$ and $b^i(-y) = -b^i(y)$ and the condition in lemma \ref{lem-precom-c2}. Then, $v_\e(y)=-v_\e(-y)$ and hence $a_{ij}$ and $b^i$ satisfies the compatibility condition.
\end{prop}
\begin{proof}
Let $\tilde{v}_\e = -v_\e(-y)$. Then $D\tilde{v}_\e = Dv_\e(-y)$ and $D^2 \tilde{v}_\e = -D^2 v_\e(-y)$. Apply it to the equation \eqref{eqn-precor1-e}. Then,
\begin{equation*} \begin{cases}
a_{ij}(y) D_{ij} \tilde{v}_\e(y) = -a_{ij}(-y) D_{ij}v_\e(-y) = 0  &\text{ in } \R^n \setminus T^{\mathfrak a} \\
b^i(y) \left( \xi^i + D_i \tilde{v}_\e(y)\right) + \e^2 \tilde{v}_\e (y) = -b^i(-y) \left( \xi^i + D_i v_\e(y)\right) - \e^2 v_\e (y) = 0 &\text{ on } \partial T^{\mathfrak a}.
\end{cases} \end{equation*}
It tells us that  $\tilde{v}_\e$ is also a solution of equation \eqref{eqn-precor1-e} and hence $v_\e(y) = \tilde{v}_\e(y) = -v_\e(-y)$ by comparison (lemma \ref{lem-precom-cp}.
From above, $v_\e$ cannot be nonnegative or nonpositive unless $v_\e = 0$ identically. So, we can conclude
\begin{equation*}
\|v_\e\|_\infty \le \osc v_\e \le C |\xi |
\end{equation*}
because of lemma \ref{lem-precor1-osc} and hence $\e^2 v_\e$ converges to 0.
\end{proof}

\section{First Corrector}\label{sec-cor1}
In this section, we are going to define the first corrector from the heuristic calculation and investigate their existence and regularity by using results in previous section.
\subsection{Existence and Regularity}
Let us consider the asymptotic expansion of $u_\e$ at $x_0 \in \Omega$. In other words, suppose that $u_\e$ has the following asymptotic expansion.
\begin{equation*}
u_\e = u_0 + \e v \left( \frac{x}{\e};\xi \right) + \e^2 w_\e \left(\frac{x}{\e}\right)+o(\e^2).
\end{equation*}

If $u_0$ is regular, then it is quite similar to the second polynomial $P(x) = \frac{1}{2} (x-x_0)^t M (x-x_0) + p \cdot (x-x_0) + u_0(x_0)$ near $x_0 \in \Omega$. So, we will identify $u_0$ with $P(x)$. Finally, define $\xi = \xi(x) = M(x-x_0) + p$ to simplify the notation. Then, by the calculation, we have first and second derivatives:
\begin{equation*} \begin{aligned}
D u_\e(x) &= M \cdot (x-x_0) + p + D_y v + \e M \cdot D_\xi v + \e Dw_\e \left( \frac{x}{\e} \right)+o(\e), \\
D^2 u_\e(x) &= M + \frac{1}{\e}D^2_y v + (\sum_{l} M^{il} D_{y_j}D_{\xi^l} v ) + (\sum_{l} D_{\xi^l}D_{y_i} v M^{lj} ) \\
&\quad + \e M D_\xi^2 v M + D^2 w_\e\left(\frac{x}{\e}\right) + o(1). \\
\end{aligned} \end{equation*}

From \eqref{eqn-main-1}, we have
\begin{equation} \label{eqn-cor} \begin{cases}
a_{ij}(y)\left(M + \frac{1}{\e}D^2_y v + \cdots + D^2 w_\e \right)^{ij} + c(r,x,y) = f(x,y) + o(1) \\
b^i(y) \cdot \left(M \cdot (x-x_0) + p + D_y v + \e M \cdot D_\xi v + \e Dw_\e \right)^i = o(\e) \\
\end{cases} \end{equation}

We can observe that, at a first line of the equation, there is one $\frac{1}{\e}$ order term. So, if $v$ and $w_\e$ exists and regular enough, then  $v$ should satisfy $a_{ij}(y) D_{ij} v(y) = 0$ in the interior of $\Omega_\e$. And, on the boundary, there are three $1$-order terms $M \cdot x$, $p$, and $D_{y} v$. Hence we could find a equation for $v$:
\begin{equation} \label{eqn-cor1} \begin{cases}
a_{ij}(y) D_{ij} v(y;\xi) = 0  &\text{ in } \R^n \setminus T^{\mathfrak a} \\
b^i(y) \cdot (\xi + D v(y;\xi))^i = 0 &\text{ on } \partial T^{\mathfrak a} \\
\end{cases} \end{equation}

As we discuss before, there is a periodic solution $v$ of the equation above if $(a_{ij})$ and $b^i$(or, our main equation \eqref{eqn-main-1}) satisfies the compatibility condition. 

Since the equation \eqref{eqn-cor1} is linear, $v(\frac{x}{\e},\xi)$ is linear with respect to $\xi$. that is, if $v^i$ is the solution of the equation \eqref{eqn-cor1} with $\xi =e^i$, then $v\left(\frac{x}{\e},\xi\right) = v^i\left(\frac{x}{\e}\right) \xi^i$. We are going to deduce  the properties of $v$ from $v^i$. We note that the solution of \eqref{eqn-main-1} is not unique since $v(y) +c$ is a solution of $v(y)$ is a solution. So, we assume that $v^i$ is the solution of the equation \eqref{eqn-cor1} when $\xi = e^i$ satisfying $v^i$ is nonnegative and $\min_{ \R^n \setminus T^{\mathfrak a}}{v^i} = 0$. 

From lemma \ref{lem-precom-osc} and lemma \ref{lem-precom-c2}, the $\mathcal{C}^2$ norm of $v^i$ is bounded by constant which is depend on the size of holls $\mathfrak{a}$. The following lemmas concerns about the relation between that constant $C$ and $\mathfrak{a}$.
\begin{lem} \label{lem-cor1-osc-uni}
Let $v$ be the periodic solution of the equation \eqref{eqn-cor1} with $\min_{ \R^n \setminus T^{\mathfrak a}}{v^i} = 0$ and assume that $\mathfrak{a}$ is small enough. Then we have
\begin{equation*}
\osc_{ \R^n \setminus T^{\mathfrak a}} v^i = \max_{ \R^n \setminus T^{\mathfrak a}} v^i \le C \cdot \mathfrak{a}.
\end{equation*}
where $C = C(n, \lambda, \Lambda, \mu)$.
\end{lem}
\begin{proof}
We are going to assume $i=1$ without losing generality. Let $Q$ be a unit cell of $\R^n$ whose center is $0$ and $B_\mathfrak{a} = Q \bigcap T^{\mathfrak a}$. Without losing any loss of generality, we may assume that the center of $Q$ and $B_\mathfrak{a}$ is 0. Then, since $v^1$ satisfies
\begin{equation} \begin{cases}
a_{ij}(y) D_{ij} v^1(y) = 0  &\text{ in } \R^n \setminus T^{\mathfrak a} \\
b^i(y) \cdot (e_1 + D v^1(y))^i = 0 &\text{ on } \partial T^{\mathfrak a}, \\
\end{cases} \end{equation}
the maximum and minimum should be achieved at a boundary point from the maximum principle (in \cite{GT}). So,
\begin{equation*} \begin{aligned}
S_0 &= \sup_{\R^n \setminus T^{\mathfrak a}} v^1 = \sup_{\partial T^{\mathfrak a}} v^1 = \sup_{\partial B_\mathfrak{a}} v^1 \\
I_0 &= \inf_{\R^n \setminus T^{\mathfrak a}} v^1 = \inf_{\partial T^{\mathfrak a}} v^1 = \inf_{\partial B_\mathfrak{a}} v^1.
\end{aligned} \end{equation*}

By the definition of $v^1$, $v^1$ is nonnegative and $I_0=0$. Let $S_1 = \sup_{\partial B_{2\mathfrak{a}}} v^1$, $I_1 = \inf_{\partial B_{2\mathfrak{a}}} v^1$ and
\begin{equation*}
h^+ = -\displaystyle\frac{1}{\mathfrak{a} \mu} ( |y|^2 - (2\mathfrak{a})^2 ) + S_1.
\end{equation*}
Then, $h^+$ satisfies
\begin{equation*} \begin{cases}
a_{ij}(y) D_{ij} h^+ \le -\displaystyle\frac{2 n \lambda}{\mathfrak{a} \mu} \le 0  &\text{ in } B_{2\mathfrak{a}} \setminus B_\mathfrak{a} \\
b^i(y) D_i h^+ = b^i(y) \left(\displaystyle\frac{2}{\mu} \nu\right)^i = 2 \ge \left( b^i(y) e^i_1\right) \ge b^i(y) D_iv^1(y) &\text{ on } \partial B_\mathfrak{a} \\
h^+ \ge v^1 &\text{ on } \partial B_{2\mathfrak{a}},
\end{cases} \end{equation*}
and then we have
\begin{equation*}
S_0 =\max_{B_{2\mathfrak{a}} \setminus B_\mathfrak{a}} v^1 \le \max_{B_{2\mathfrak{a}} \setminus B_\mathfrak{a}} h^+ = S_1 + \displaystyle \frac{\mathfrak{a}}{2 \mu_1}
\end{equation*}
from the comparison.
Similarly,  we can show
\begin{equation*}
I_0 \ge I_1 - \displaystyle \frac{\mathfrak{a}}{2 \mu}.
\end{equation*}

Let $\widetilde{v}(z) = v^1(\mathfrak{a}z)$. Then $\widetilde{v}$ is nonnegative and satisfies
\begin{equation*}
a_{ij}(\mathfrak{a}z)D_{ij}\widetilde{v}(z) = 0 \text{ in } B_3(0) \setminus B_1(0)
\end{equation*}
whenever $\mathfrak{a} < \frac{1}{6}$.

By applying the Harnack estimate (in \cite{CC}) on $\widetilde{v}$ in $B_3$, we have
\begin{equation*}
\sup_{\partial B_2(0)} \widetilde{v} \le C(n,\lambda, \Lambda) \inf_{\partial B_2(0)} \widetilde{v},
\end{equation*}
which implies  $S_1 \le C(n,\lambda, \Lambda) I_1$.

Now combining these three results, we have
\begin{equation*} \begin{aligned}
S_0 &\le S_1 + \displaystyle \frac{\mathfrak{a}}{2 \mu} \le C(n,\lambda, \Lambda) I_1 + \displaystyle \frac{\mathfrak{a}}{2 \mu}\\
&\le C(n,\lambda, \Lambda) \left( I_0 + \displaystyle \frac{\mathfrak{a}}{2 \mu} \right) + \displaystyle \frac{\mathfrak{a}}{2 \mu} \le C(n,\lambda, \Lambda, \mu) \mathfrak{a}.
\end{aligned} \end{equation*}
\end{proof}

\begin{lem}[Interior estimate of $Dv$] \label{lem-cor1-dv1}
Let $v^i(y)$ be the solution of the equation \eqref{eqn-cor1} with $\xi = e^i$. And suppose that the coefficient functions $(a_{ij})$ and $b^i$ satisfies $\| a_{ij} \|_{C^\alpha(\R^n \setminus T^{\mathfrak a})} + \| b^i(\mathfrak{a}~\cdot) \|_{ C^{1,\alpha}(\frac{1}{\mathfrak{a}}(\R^n \setminus T^{\mathfrak a}))} \le \Lambda$. Then, we have the following estimate
\begin{equation*}
d(y) |Dv^i(y)| + d(y)^2 |D^2v^i(y)| + \min\left( d(y_1),d(y_2) \right)^{2+\alpha} \displaystyle\frac{ \left| D^2 v^i(y_1) - D^2 v^i(y_2) \right| }{|y_1- y_2|} \le C \mathfrak{a}
\end{equation*}
where $C=C(n, \lambda, \Lambda, \mu)$ is the same as in \ref{lem-cor1-osc-uni} and $d(y) = d(y,T^{\mathfrak a})$ is a distance between $y$ and $T^{\mathfrak a}$. In particular,
\begin{equation*}
|Dv^i(y)| + \mathfrak{a} |D^2v^i(y)| + \mathfrak{a}^{1+\alpha} \displaystyle\frac{ \left| D^2 v^i(y_1) - D^2 v^i(y_2) \right| }{|y_1- y_2|} \le C
\end{equation*}
if $y, y_1, y_2 \in \partial B_{2\mathfrak{a}}$.
\end{lem}
\begin{proof}
It follows from lemma \ref{lem-cor1-osc-uni} and the standard interior $\mathcal{C}^{2,\alpha}$ estimate(See chapter 6 of \cite{GT}).
\end{proof}

\begin{lem}[Global estimaete of $Dv$] \label{lem-cor1-dv2}
Let $v^i(y)$ be the solution of the equation \eqref{eqn-cor1} with $\xi =e^i$. And suppose that the coefficient functions $(a_{ij})$ and $b^i$ satisfies $\| a_{ij} \|_{C^\alpha(\R^n \setminus T^{\mathfrak a})} + \| b^i(\mathfrak{a}~\cdot) \|_{ C^{1,\alpha}(\frac{1}{\mathfrak{a}}(\R^n \setminus T^{\mathfrak a}))} \le \Lambda$. Then, we have
\begin{equation}
\| Dv^i(y) \|_{L^\infty(\R^n \setminus T^{\mathfrak a})} \le C_1
\end{equation}
where $C_1=C_1(n, \lambda, \Lambda, \mu)>0$.
\end{lem}
\begin{proof}
Let $Q$ be a unit cell of $\R^n$ whose center is $0$ and $B_\mathfrak{a} = B_\mathfrak{a}(0) = Q \cap T^{\mathfrak a}$. We will show $\sup_{Q \setminus B_\mathfrak{a}} |D v^i|$ is bounded. From lemma \ref{lem-cor1-dv1}, $|Dv^i|$ in bounded in $Q \setminus B_{2\mathfrak{a}}$, so we just need to show that the gradient is bounded in $B_{2\mathfrak{a}} \setminus B_\mathfrak{a}$. Let us define the scaled function $\widetilde{v} =\displaystyle\frac{1}{\mathfrak{a}} v^i(\mathfrak{a}z)$. Then $\widetilde{v}$ satisfies the following equation
\begin{equation*} \begin{cases}
a_{ij}(\mathfrak{a}z) D_{ij} \widetilde{v}(z) = 0 &\text{ in } B_2 \setminus B_1 \\
b^i(\mathfrak{a}z) \left( e^i + D \widetilde{v}(z) \right)^i  = 0 &\text{ on } \partial B_1 \\
\widetilde{v}_\e = \phi(z) &\text{ on } Q \setminus B_2. \\
\end{cases} \end{equation*}
where $\phi(z) = \displaystyle\frac{1}{\mathfrak{a}} v^i(\mathfrak{a}z)$.

Since $[a_{ij}(\mathfrak{a}~\cdot)]_{C^\alpha(B_2 \setminus B_1)} = \mathfrak{a}^\alpha [a_{ij}(\cdot)]_{C^\alpha(B_{2\mathfrak{a}} \setminus B_\mathfrak{a})}$, $[a_{ij}(\mathfrak{a}~\cdot)]_{C^\alpha(B_2 \setminus B_1)} + [b^i(\mathfrak{a}~\cdot)]_{C^{1,\alpha}(B_2 \setminus B_1)} \le \Lambda$. Additionally, since $[\phi]_{C^{2,\alpha}(\partial B_2)}$ is bounded independently on $\mathfrak{a}$ from lemma \ref{lem-cor1-dv1}, we have the following estimate
\begin{equation*}
[\widetilde{v}]_{C^{2,\alpha}(B_2 \setminus B_1)} \le C(n,\Lambda,\lambda,\mu) \left( [b^i(\mathfrak{a}~\cdot)]_{C^{1,\alpha}(\partial B_2)} + [\phi]_{C^{2,\alpha}(\partial B_2)} \right) \le C_1(n,\Lambda,\lambda,\mu)
\end{equation*}
by using the estimate in \cite{GT}.
Especially, $|D\widetilde{v}(z)| = |Dv^i(y)|$ is bounded by $C_1$ whenever $y \in B_{2\mathfrak{a}} \setminus B_\mathfrak{a}$.
\end{proof}

\section{Second Corrector and Uniformly Ellipticity of $\overline{L}$}\label{sec-cor2 and ellipticity}
In this section, we define the effective equation $\overline{L}$ by finding the second corrector.
And, we prove two important properties of $\overline{L}$: the uniform ellipticity and continuity of the effective equation $\overline{L}$. Throughout this section, we assume that \eqref{eqn-main-1} satisfies the compatibility condition and condition I in chapter 1 hold.
\subsection{The Existence of Second Corrector and Effective Equation.}
Let us define $V(y)$ as a $\R^n$-valued function whose components are $v^i(y)$ and $\xi = \xi(x) = M(x-x_0) +p$ for a given vector $p \in \R^n$ and a symmetric matrix $M$. Additionally, let us define a matrix $Z(y,M)$ as
\begin{equation} \label{eqn-cor2-bddz}
Z_{ij}(y,M) = \sum_{l} M^{il} D_j v^l(y) + \sum_{l}D_i v^l(y) M^{lj}. \\
\end{equation}

From lemma \ref{lem-cor1-dv2}, we can deduce the following lemma.
\begin{lem} \label{lem-cor2-bddz}
For any given a symmetric matrix $M$ and a point $y$ in $\R^n \setminus T^{\mathfrak a}$, the following estimate holds:
\begin{equation*}
\left| Z(y,M)\right| \le C(n,\lambda,\Lambda,\mu) \| M \|.
\end{equation*}
\end{lem}

Now, let us apply $v(y,\xi) = V(y) \cdot \xi(x)$ to the equation \eqref{eqn-cor} . Then we have the following:
\begin{equation} \label{eqn-cor2} \begin{cases}
a_{ij}(y)\left( M + Z(y,M) + D^2w_\e(y) \right)_{ij}+c(r,x,y) = f(x,y) + o(1) \\
b^i(y) \cdot \left( \sum_{l} M^{il} v^l + D_i w_\e \right) = o(1).  \\
\end{cases} \end{equation}

We note that our second corrector $w_\e$ should satisfy the above equation. We add the auxiliary term $-\e^2 w_\e$  to the interior equation and $\e^2 w_\e$ to the boundary equation to guarantee the existence of second corrector. Then, we have the following equation about $y$ variable for fixed $x=x_0$ and $r=r_0$ .
\begin{equation} \label{eqn-cor2e} \begin{cases}
-\e^2 w_\e(y) + a_{ij}(y)\left( M + Z(y,M) + D^2w_\e(y) \right)_{ij}+c(r_0,x_0,y) = f(x_0,y) &\text{ in } \R^n \setminus T^{\mathfrak a} \\
b^i(y) \cdot \left( \sum_{l} M^{il}v^l + D_i w_\e(y) \right)+\e^2 w_\e(y) = 0 &\text{ on } \partial T^{\mathfrak a}
\end{cases} \end{equation}

From lemma \ref{lem-cor2-bddz} $Z(y,M)$ is bounded. So, the equation \eqref{eqn-cor2e} is well defined. And, by adding the auxiliary term, we can find a bounded viscosity solution for each $\e$ and we also can prove the comparison principle like Lemma \ref{lem-precom-cp}. Since the proof of comparison principle is  similar to that of lemma \ref{lem-precom-cp}, we just state it without proof.
\begin{lem}(Comparison) \label{lem-cor2-com}
Suppose that $w^+$ is a super-solution of \eqref{eqn-cor2e} and $w^-$ is a sub-solution of \eqref{eqn-cor2e} for fixed $M$, $a$, $r_0$, $x_0$ and $\e$. Then we have
\begin{equation*}
w^+ \ge w^-.
\end{equation*}
\end{lem}

\begin{lem} \label{lem-cor2-exist}
For each $M$, $\mathfrak{a}$, $r_0$, $x_0$ and $\e$, there is a periodic solution $w_\e(y;M,r_0,x_0)$ of the equation \eqref{eqn-cor2e} satisfying
\begin{equation*}
\| \e^2 w_\e \|_{L^\infty(\R^n \setminus T^{\mathfrak a})} \le C \left( \|M\| + \|c(r_0,x_0,\cdot)\|_{L^\infty(\R^n \setminus T^{\mathfrak a})} + \| f(x_0,\cdot) \|_{L^\infty(\R^n \setminus T^{\mathfrak a})} \right)
\end{equation*}
where $C$ is a constant depending only on $n$, $\lambda$, $\Lambda$, $\mu$ and $\mathfrak{a}$.
\end{lem}
\begin{proof}

For fixed $\mathfrak{a}$, let us define
\begin{equation*}
K = n \Lambda \sum_{ij} \left( \| M \| + \| Z_{ij}(\cdot,M) \|_\infty + \|c(r_0,x_0,\cdot)\|_\infty + \| f(x_0,\cdot) \|_\infty +2 \| V \|_\infty \| M \|  \right).
\end{equation*}
Then, from lemma \ref{lem-cor2-bddz} and \ref{lem-cor1-osc-uni}, we have
\begin{equation} \begin{aligned}
K &\le C(n,\lambda,\Lambda) \left( \| M \| + \sum_{ij}\| Z_{ij}(\cdot,M) \|_\infty + \|c(r_0,x_0,\cdot)\|_\infty + \| f(x_0,\cdot) \|_\infty +2 \| V \|_\infty \| M \|  \right) \\
&\le C(n,\lambda,\Lambda,\mu) \left( \|M\| + \|c(r_0,x_0,\cdot)\|_{L^\infty(\R^n \setminus T^{\mathfrak a})} + \| f(x_0,\cdot) \|_{L^\infty(\R^n \setminus T^{\mathfrak a})} \right).
\end{aligned} \end{equation}
And, from the definition of $K$, $w^+ = \frac{1}{\e^2} K$ and $w^- = - \frac{1}{\e^2} K$ are super and sub-solution of \eqref{eqn-cor2e} respectively.
So, from the similar reason in lemma \ref{lem-precom-cp}, we can find the solution $w_\e$ of \eqref{eqn-cor2e} which satisfies
\begin{equation*}
\| \e^2 w_\e \|_{L^\infty(\R^n \setminus T^{\mathfrak a})} \le C(n,\lambda,\Lambda,\mu ) \left( \|M\| + \|c(r_0,x_0,\cdot)\|_{L^\infty(\R^n \setminus T^{\mathfrak a})} + \| f(x_0,\cdot) \|_{L^\infty(\R^n \setminus T^{\mathfrak a})} \right).
\end{equation*}
\end{proof}

\begin{lem} \label{lem-cor2-osc}
The solution $w_\e$ of the equation \eqref{eqn-cor2e} satisfies
\begin{equation}
\osc w_\e \le C \left( \|M\| + \|c(r_0,x_0,\cdot)\|_{L^\infty(\R^n \setminus T^{\mathfrak a})} + \| f(x_0,\cdot) \|_{L^\infty(\R^n \setminus T^{\mathfrak a})} \right).
\end{equation}
where $C$ depends only on $n$, $\lambda$, $\Lambda$ and $\mu$.
\end{lem}

\begin{proof}
Let $\widetilde{f}(y) = f(x_0,y) + \e^2 w_\e(y) - a_{ij}(y) \left( M + Z(y,M) \right)_{ij} - c(r_0,x_0,y)$ and $\widetilde{g}(y) = b^i(y) \cdot M^{il} v^l (y) - \e^2 \min_{\R^n \setminus T^{\mathfrak a}}w_\e$. Then $\widehat{w}_\e = w_\e -\min_{\R^n \setminus T^{\mathfrak a}} w_\e$ satisfies the following equation:
\begin{equation*} \begin{cases}
a_{ij}(y) D_{ij} \widehat{w}_\e(y) = \widetilde{f}(y) &\text{ in } \R^n \setminus T^{\mathfrak a} \\
b^i(y) D_i \widehat{w}_\e(y) + \e^2 \widehat{w}_\e (y)  = \widetilde{g}(y) &\text{ on } \partial T^{\mathfrak a}. \\
\end{cases} \end{equation*}

From lemma \ref{lem-cor2-bddz}, lemma \ref{lem-cor2-exist} and lemma \ref{lem-cor1-osc-uni}, we have
\begin{equation}
|\widetilde{f}(y)| + |\widetilde{g}(y)| \le C(n,\lambda,\Lambda,\mu) \left( \|M\| + \|c(r_0,x_0,\cdot)\|_{L^\infty(\R^n \setminus T^{\mathfrak a})} + \| f(x_0,\cdot) \|_{L^\infty(\R^n \setminus T^{\mathfrak a})} \right).
\end{equation}

We note that $\widetilde{f}$ and $\widetilde{g}$ are bounded uniformly on $0 < \e \le 1$. Therefore, the oscillation of $w_\e$ is bounded because of lemma \ref{lem-precom-osc}.
\end{proof}

Let $\sigma_\e = \min_{\R^n \setminus T^{\mathfrak a}} \e^2 w_\e$ and $\widehat{w}_\e = w_\e - \frac{1}{\e^2} \sigma_\e$. Then, $\widehat{w}_\e = w_\e - \frac{1}{\e^2} \sigma_\e$ satisfies
\begin{equation*} \begin{cases}
-\e^2 \widehat{w}_\e(y) + a_{ij}(y) D_{ij} \widehat{w}_\e(y) = \widehat{f}(y) &\text{ in } \R^n \setminus T^{\mathfrak a} \\
b^i(y) D_i\widehat{w}_\e(y) + \e^2 \widehat{w}_\e(y) = \widehat{g}(y) &\text{ on } \partial T^{\mathfrak a} \\
\end{cases} \end{equation*}
where $\widehat{f}(y) = f(x_0,y) +\sigma_\e - a_{ij}(y) \left(M + Z(y,M) \right)_{ij} - c(r_0,x_0,y)$\\
 and $\widehat{g}(y) = b^i(y) \cdot M^{il} v^l_\e(y) - \sigma_\e $.

From proposition \ref{prop-exist-alpha}, $v^l(y)$ is in $\mathcal{C}^{1,\alpha}(\R^n \setminus T^{\mathfrak a})$ and hence $Z(y,M)$ is in $\mathcal{C}^\alpha (\R^n \setminus T^{\mathfrak a})$ for fixed  $M$, $r_0$, and $x_0$. And then from the condition \eqref{con-alpha} in chapter 1, $|\widehat{f}|_{C^\alpha} + |\widehat{g}|_{C^{1,\alpha}}$ is bounded uniformly on $\e$.  So, we have $\mathcal{C}^{2,\alpha}$ estimate for $\widehat{w}_\e$.
\begin{cor}
Let $\widehat{w}_\e(y) = \widehat{w}_\e(y,M,r,x_0)$ be the solution of the equation of \eqref{eqn-cor2e} which satisfies the condition \eqref{con-alpha} (in condition I). Then, $ \| \widehat{w}_\e(y) \|_{C^{2,\alpha}} $ is bounded uniformly on $\e$ for given any $M$, $r$ and $x_0$.
\end{cor}
The proof of corollary above is almost same as that of lemma \ref{lem-precom-c2}. So we omit the proof.

\begin{cor}
There is a unique limit of $\e^2 w_\e$ as $\e \ra 0$.
\end{cor}
\begin{proof}
From lemma \ref{lem-cor2-exist} and lemma \ref{lem-cor2-osc}, there exists a subsequence $\e_k$ such that $\e_k^2 w_{\e_k}$ converges to a constant $\sigma$. And, from similar argument in lemma \ref{lem-alpha-uniq}, $\sigma$ should be same even though we change the subsequence because of the uniform $\mathcal{C}^{2,\alpha}$ estimate of $\widehat{w}_\e$. That implies $\e^2 \widehat{w}_\e$ converges to $\sigma$.
\end{proof}

\begin{definition}
$\overline{L}(M,r_0,x_0)$ is the limit of $\e^2 w_\e(y)$ for fixed $M$, $r_0$ and $x_0$.
\end{definition}
We prove later that the limit equation of $u_\e$ satisfies the equation $\overline{L}(M,r,x) = 0$ in chapter \ref{sec-Homogenization}. Usually, it is called as an {\it Effective equation}.

\subsection{Uniformly Ellipticity and Continuity of $\overline{L}$}
We will end this section by proving two important properties of $\overline{L}$, {\it uniformly ellipticity} and {\it continuity}.

\begin{thm} \label{thm-cor2-elliptic}
Assume the conditions in condition I hold and the equation \eqref{eqn-main-1} satisfies the compatibility condition.
Then, there is a positive real number ${\mathfrak a}_0$ depending only on $n$, $\lambda$, $\Lambda$ and $\mu$ such that if the size of hole $\mathfrak{a}$ is less than or equal to ${\mathfrak a}_0$, then $\overline{L}$ is uniformly elliptic. In other words, there is a positive constant $\overline{\lambda}=\overline{\lambda}({\mathfrak a}_0)$ satisfying $\overline{L}(M+N,r,x) \ge \overline{L}(M,r,x) + \overline{\lambda} \| N \|$ for any symmetric matrix $M$ and positive matrix $N$.
\end{thm}
\begin{proof} \item
We will show $\overline{L}(M+N,r_0,x_0) - \overline{L}(M,r_0,x_0) \ge \overline{\lambda} \| N \|$ for any given $M$, $N$, $r_0$, and $x_0$. Actually, it is equivalent to prove $\e^2 w_\e(y;M+N,r_0,x_0) - \e^2 w_\e(y;M,r_0,x_0) \ge \overline{\lambda} \| N \| + o(1)$. So, we first define $\widetilde{w}_\e = w_\e(y;M+N,r_0,x_0) - w_\e(y;M,r_0,x_0)$. Then, $\widetilde{w}_\e$ satisfies the following equation.
\begin{equation} \label{eqn-cor2e-elliptic} \begin{cases}
-\e^2 \widetilde{w}_\e(y) + a_{ij}(y) D_{ij}\widetilde{w}_\e(y) = - a_{ij}(y) \left( N + Z(y,N) \right)_{ij} &\text{ in } \R^n \setminus T^{\mathfrak a} \\
b^i(y) D_i \widetilde{w}_\e(y) +\e^2 \widetilde{w}_\e(y) = - b^i(y) N^{ij} v^j(y) &\text{ on } \partial T^{\mathfrak a}. \\
\end{cases} \end{equation}

We will construct a (viscosity) sub-solution $h_\e(y)$ such that $\e^2 h_\e$ converges to a positive constant. We consider the case $\| N\| =1$ because the general result can be obtained by scaling. To construct a sub-solution, we need to estimate the righthand side of the equation.
First, from lemma \ref{lem-cor2-bddz}, $|Z(y,N)|_{L^\infty} \le \| N \| | D_y V_\e(\cdot,x_0) |_\infty$. Hence we have
\begin{equation*}
a_{ij}(y) \left( N_{ij} + Z_{ij}(y,N) \right)
\le - \lambda \| N \| + \Lambda \| N \| \sum_{ij} |D_i v^j| + o(1)
\end{equation*}
for all $y \in \R^n \setminus T^{\overline{a}}$.

So, for small $\e$, we have
\begin{equation*}
a_{ij}(y) \left( N_{ij} + Z_{ij}(y,N) \right)
\le - \displaystyle\frac{2 \lambda}{3} + \Lambda \sum_{ij} |D_i v^j|.
\end{equation*}

From lemma \ref{lem-cor1-dv2}, $\sum_{ij} |D_i v^j_\e|$ is bounded uniformly on $a$. More precisely,
\begin{equation*}
\Lambda \sum_{ij} |D_i v^j| \le n^2 \Lambda C_1
\end{equation*}
where $C_1$ is a constant in lemma \ref{lem-cor1-dv2} that depends only on $n$, $\lambda$, $\Lambda$, and $\mu$.

Secondly, from lemma \ref{lem-cor1-dv1}, $\sum_{ij} |D_i v^j(y)|$ is small if $y$ is far from the boundary. More precisely, for any given $\overline{\mathfrak{a}}$, if the size of halls $\mathfrak{a}$ is less than or equal to ${\mathfrak a}_0 = \min\left\{\displaystyle\frac{ \lambda}{12 n^2 \Lambda C},\frac{1}{2} \right\}\overline{\mathfrak{a}}$, then we have
\begin{equation*}
\Lambda \sum_{ij} |D_i v^j(y)| \le \displaystyle\frac{n^2 \Lambda C_2 \mathfrak{a}}{d(y,T^{\mathfrak a})} \le \displaystyle\frac{2n^2 \Lambda C_2 {\mathfrak a}_0}{\overline{\mathfrak{a}}/2} \le\displaystyle\frac{\lambda}{3}
\end{equation*}
where $d(y,T^{\mathfrak a})$ is a distant between $y$ and $T^{\mathfrak a}$ and the constant $C_2$ is same in lemma \ref{lem-cor1-dv1}.

Finally, from lemma \ref{lem-cor1-osc-uni}, we have
\begin{equation*}
b^i(y) N_{ij} v^j(y) \ge -C_3 \mathfrak{a}
\end{equation*}
where $C_3$ is a constant which is independent of $\e$ and $\mathfrak{a}$.

Now we are ready to define the barrier. Let us define the function $h$ as follow.
\begin{equation*}
h(y) =
\begin{cases}
\displaystyle\frac{K}{2}(|y|+\overline{\mathfrak{a}})^2 &\text{ ,if } \,0 \le |y| \le \overline{\mathfrak{a}} \\
\displaystyle\frac{\beta}{2}\left(|y|-\frac{10K}{\beta}\overline{\mathfrak{a}}\right)^2 + \displaystyle\left( 2K - \frac{\beta \overline{\mathfrak{a}}^2 }{2} \left( 1 - \frac{10K}{\beta} \right)^2 \right) &\text{ ,if } \,\overline{\mathfrak{a}} \le |y|
\end{cases}
\end{equation*}
We will select $K$, $\beta$, and $\overline{\mathfrak{a}}$ later. Then, the function $h$ is continuous and twice differentiable except the points on $\partial B_{\overline{\mathfrak{a}}}$. And from the calculation, we have
\begin{equation*}
\begin{cases}
a_{ij}(y) D_{ij} h(y) \ge n \lambda K &\text{ if } \mathfrak{a} < |y| < \overline{\mathfrak{a}} \\
a_{ij}(y) D_{ij} h(y) \ge -n \Lambda \beta &\text{ if } \overline{\mathfrak{a}} < |y|.
\end{cases}
\end{equation*}
And, on the boundary $\partial B_\mathfrak{a}$, we have
\begin{equation*}
b^i(y) Dh(y) = - b^i(y) K \left( \mathfrak{a} + \overline{\mathfrak{a}} \right) \nu \le - K \overline{\mathfrak{a}} \mu.
\end{equation*}

If we choose $K$ bigger than $\beta$, then our function $h(y)$ has sharp edge on $\partial B_{\overline{\mathfrak{a}}}$, and hence there are no second order polynomials touching $h({y})$ by above at any points on $\partial B_{\overline{\mathfrak{a}}}$. Select $\beta = \displaystyle\frac{\lambda}{3n\Lambda}$, $K = \displaystyle\frac{n \Lambda C_1}{\lambda} + \displaystyle\frac{3 C_3}{\mu} + 1 + \beta$. Then, $K > \beta$ and,
\begin{equation} \begin{cases}\label{eqn-cor2-elliptic-1}
a_{ij}(y) D_{ij} h(y) \ge n \lambda K \ge n^2 \Lambda C_1 \ge -\displaystyle\frac{2 \lambda}{3} + \Lambda \sum_{ij} |D_i v^j_\e| &\text{ if } \mathfrak{a} < |y| < \overline{\mathfrak{a}}, \\
a_{ij}(y) D_{ij} h(y) \ge -n \Lambda \beta \ge -\displaystyle\frac{\lambda}{3} \ge - \displaystyle\frac{2 \lambda}{3} + \Lambda \sum_{ij} |D_i v^j_\e| &\text{ if } \overline{\mathfrak{a}} < |y|.
\end{cases} \end{equation}

And, on the boundary, we have
\begin{equation} \label{eqn-cor2-elliptic-2}
b^i(y) Dh(y) \le - K \overline{a} \mu \le - 3 C_3 \overline{\mathfrak{a}}.
\end{equation}

Finally, select $\overline{\mathfrak{a}}$ satisfying $\displaystyle\frac{10K}{\beta}\overline{\mathfrak{a}} < 1$ and $\overline{\lambda} = C_3 \overline{\mathfrak{a}}$. Then, $h'(r) < 0$ if $r \ge1$ and hence
\begin{equation*}
h_\e(y) = \max_{m \in \Z^n} h(y-m) + \displaystyle\frac{1}{\e^2} \overline{\lambda}
\end{equation*}
is a sub-solution of equation \eqref{eqn-cor2e-elliptic} for small $\e > 0$ from \eqref{eqn-cor2-elliptic-1} and \eqref{eqn-cor2-elliptic-2}.

And, from lemma \ref{lem-cor2-com}, we have
\begin{equation*}
\e^2 \widetilde{w}_\e \ge \e^2 h_\e(y) \ge \e^2 \max_{m \in \Z^n} h(y-m) + \overline{\lambda} \ra \overline{\lambda} \text{ as } \e \ra 0.
\end{equation*}
That is
\begin{equation*}
\overline{L}(M+N,r_0,x_0) -\overline{L}(M,r_0,x_0) \ge \overline{\lambda} > 0.
\end{equation*}
\end{proof}

\begin{remark}
If $a_{ij}(y) = \delta_{ij}$(Laplace equation), $c(r,x,y) =0$ and $b^i(y) = \nu^i$(Neumann boundary condition) in equation \eqref{eqn-cor2e}, then we can prove the uniform ellipticity even the size of hall $\mathfrak{a}$ is large because we can use the divergence theorem. Let $Q$ be a one cell whose center is 0 and punctured by a ball $B_\mathfrak{a}(0)$. We identify $\e^2w_\e(y;M,r_0,x_0)$ with  $\overline{L}(M,r_0,x_0)$ because the error between them is of order $o(\e)$. Then, we have the followng by using the divergence theorem:
\begin{equation*} \begin{aligned}
 |Q \setminus B_\mathfrak{a}| &\overline{L} (M,r_0,x_0) = \int_{Q \setminus B_\mathfrak{a}} \overline{L} (M,r_0,x_0) dy = \int_{Q \setminus B_\mathfrak{a}} \e^2 w_\e(y) dy \\
 &= \int_{Q \setminus B_\mathfrak{a}} tr(M) +  \sum_i Z_{ii}(y,M) + \triangle w_\e + c(r_0,x_0,y) - f(x_0,y) dy \\
 &= |Q \setminus B_\mathfrak{a}| tr(M) + \int_{Q \setminus B_\mathfrak{a}} \sum_i M^{il} D_i v^l_\e dx + \triangle w_\e dy + \langle c \rangle - \langle f \rangle \\
 &= |Q \setminus B_\mathfrak{a}| tr(M) + \int_{\partial B_\mathfrak{a}} \sum_i M^{il} (v^l_\e(y) -v^l_\e(0)) \nu^i + D_i w_\e \nu^i d \sigma_y + \langle c \rangle - \langle f \rangle \\
 &= |Q \setminus B_\mathfrak{a}| tr(M) - |\partial B_\mathfrak{a}| \overline{L}(M,r,x_0) + \langle c \rangle - \langle f \rangle \\
\end{aligned} \end{equation*}
\end{remark}
where $\langle c \rangle = \int_{Q \setminus B_a} c(r_0,x_0,y) dy$ and $\langle f \rangle = \int_{Q \setminus B_a} f(x_0,y) dy$.
So, we have the explicit formula of $\overline{L}(M,r_0,x_0)$:
\begin{equation*}
\overline{L}(M,r_0,x_0) = \displaystyle\frac{|Q \setminus B_a|}{|Q \setminus B_a| + |\partial B_a|} tr(M) + \displaystyle\frac{1}{|Q \setminus B_a| + |\partial B_a|}\left( \langle c \rangle - \langle f \rangle \right)
\end{equation*}
And the uniform ellipticity comes automatically from above formula.

\begin{prop} \label{prop-cor2-conti}
Assume the conditions in condition I and the equation \eqref{eqn-main-1} satisfies the compatibility condition.
Then, \item
\begin{enumerate}
\item
$\overline{L}(M,r,x)$ is continuous with respect to $r$ and $x$ variables.
\item
$\overline{L}(M,r,x)$ is non-incresing with $r$ variable.
\end{enumerate}
\end{prop}
\begin{proof} \item
\begin{enumerate}
\item
We will show that $\overline{L}$ is continuous with $x$ variable for fixed $M$, $r$. And we omit the proof of the continuity with $r$ because that is quite similar to the proof of continuity with $x$. Now, suppose that $M$, $r$ are fixed. And let
\begin{equation*} \begin{aligned}
\widetilde{f}(x,y) &= f(x,y) - c(r,x,y), \text{ and } \\
\widetilde{g}(y) &= - b^i(y) \cdot \left( M  V(y) \right)^i.
\end{aligned} \end{equation*}
Then, the equation for second corrector can be modified to
\begin{equation*} \begin{cases}
-\e^2 w_\e(y) a_{ij}(y) D_{ij} w_\e(y) = \widetilde{f}(x,y) &\text{ in } \R^n \setminus T^{\mathfrak a} \\
b^i(y) D_iw_\e(y) + \e^2 w_\e(y) = \widetilde{g}(y) &\text{ on } \partial T^{\mathfrak a} \\
\end{cases} \end{equation*}

For the simplicity of notation, we define $w_\e(y;x) = w_\e(y;M,r,x)$ and $\widetilde{w}_\e(y) = w_\e(y,x_1) - w_\e(y,x_2)$ for some $x_1$, $x_2 \in \overline{\Omega}$. Then, $\widetilde{w}_\e(y)$ satisfies the following equation:
\begin{equation}
\begin{cases}
-\e^2 \widetilde{w}_\e(y) + a_{ij}(y) D_{ij} \widetilde{w}_\e(y) = \widetilde{f}(x_1,y)- \widetilde{f}(x_2,y) &\text{ in } \R^n \setminus T^{\mathfrak a} \\
b^i(y) D_i\widetilde{w}_\e(y) + \e^2 \widetilde{w}_\e(y) = 0 &\text{ on } \partial T^{\mathfrak a}. \\
\end{cases}
\end{equation}

Since $f$ and $c$ are continuous uniformly on $y$,
\begin{equation*}
| f(x_1,y) - f(x_2,y) | + | c(r,x_1,y) - c(r,x_2,y) | \le \sigma( |x_1 - x_2| )
\end{equation*}
where $\sigma : \R^+ \ra\R^+$ is a nondecreasing function with $\lim_{r\ra0+} \sigma(r) =0$.

Hence, $| \e^2 \tilde{w}_\e | \le \sigma(|x_1-x_2|) $ because of lemma \ref{lem-cor2-exist}. And the conclusion comes by taking limit on both side.
\item
It can be shown by using similar argument above and the comparison principle.
\end{enumerate}
\end{proof}

\section{Homogenization }\label{sec-Homogenization}
\subsection{Proof of theorem \ref{thm-main-1}}
In this section, we are going to prove the limit of solutions satisfies the homogenized equation.
First, assume that $u_\e$ is bounded uniformly on $\e$. Then we can define the limit of $u_\e$ in the following way.
\begin{definition} \label{def-lim-u}
Define $u^*$ and $u_*$ as follow:
\begin{equation} \label{eqn-lim-u}\begin{aligned}
u_*(x) &= \lim_{\e\ra 0} \inf \{u_{\e\rq{}} (x_{\e\rq{}}):\,  x_{\e\rq{}}\in B_{\e\rq{}}(x)\cap\Omega_{\e\rq{}} , 0<\e\rq{}\leq \e\} \\
u^*(x) &= \lim_{\e\ra 0} \sup \{u_{\e\rq{}} (x_{\e\rq{}}):\,  x_{\e\rq{}}\in B_{\e\rq{}}(x)\cap\Omega_{\e\rq{}} , 0<\e\rq{}\leq \e\}.
\end{aligned} \end{equation}
\end{definition}

We will prove that $u_*$ is a super-solution of the equation of the effective equation(equation \ref{eqn-rst-1}). And the lower semi-continuity of $u_*$(upper semi-continuity of $u^*$) comes from a similar argument as in \cite{CIL}.

Now we are going to prove our main theorem.
\begin{proof}[\bf{Proof of theorem \ref{thm-main-1}}] \item
Suppose that $u_*$ is not a viscosity super-solution. Then, there is a second polynomial $P(x)$ touches $u_*(x)$ from below at $x_0 \in \overline{\Omega}$ such that there exists $R_0$ satisfying $u_*(x) \ge P(x)$ in $B_{R_0}(x_0)$ and $u_*(x_0) = P(x_0)$ and $\overline{L}(D^2P(x_0),P(x_0),x_0) \ge 6 \eta >0$.

For the simplicity, suppose that $x_0=0$ and $u_*(x_0) =0$. Set $P^\delta(x) = P(x) - \delta |x|^2$. Then, since $B_R(0) \setminus T_\e$ is compact for any given $0 < R \le R_0$, we can find $\widehat{x}(\e) \in B_R(0) \setminus T_\e$ which satisfies
\begin{equation}
u_\e(\widehat{x}(\e)) - P ^\delta(\widehat{x}(\e)) = \min_{B_R(0) \setminus T_\e} \left(u_\e(x) - P ^\delta(x)) \right).
\end{equation}

From the definition of $u_*$, there is a subsequence $\{(\e_n,x_n)\} \in (0,1] \times \overline{\Omega}$ which converges to $(0,0)$ satisfying $x_n\in \Omega_{\e_n}$ and

\begin{equation}
\lim_{n \ra \infty} u_{\e_n}(x_n) = u_*(0).
\end{equation}

Set  $A = \{ \e_n \}$. Since $x_n \in B_R(0)$ and $B_R(0)$ is compact, We can find a subsequence of $\{\widehat{x}_n = \widehat{x}(\e_n)\}$ which converges to some $y \in \overline{B}_R(0)$ as $\e_n \ra 0$. And hence we assume that
$\lim_{n \ra 0} u_{\e_n}(x_n) = u_*(0)$ and $\widehat{x}(\e_n) \ra y$.

Since $u_{\e_n}(\widehat{x}_n) - P^\delta(\widehat{x}_n) \le u_{\e_n}(x_n) - P^\delta(x_n)$ from the definition of $\widehat{x}_n$, by taking limit infimun on both side, we have
\begin{equation}
u_*(y) - P^\delta(y) \le \liminf_{n \ra \infty} \left( u_{\e_n}(\widehat{x}_n) - P^\delta(\widehat{x}_n) \right) \le \liminf_{n \ra \infty} \left( u_{\e_n}(x_n) - P^\delta(x_n) \right) \le u_*(0) - P^\delta(0) = 0.
\end{equation}
But, by the definition of $P^\delta$, we have $$\delta |y|^2 \le u_*(y) - P(y) + \delta |y|^2 =u_*(y) - P^\delta(y)\leq  u_*(0) - P^\delta(0) = 0.$$  Therefore $\delta |y|^2 \le 0$ and hence $y=0$. That implies the sequence $\{\widehat{x}_n\}$ converges to 0 (not as a subsequence). And, from above inequality, we also conclude that
\begin{equation}
\liminf_{n \ra \infty} u_{\e_n}(\widehat{x}_n) = u_*(0)
\end{equation}

Let $\delta_1 = 1$. Since the sequence $\{ \widehat{x}(\e_n) ; \e_n \in A\}$ related with $\delta_1$ converges to 0 as $n \ra \infty$, we can find $\e_1 \in A$ and $\widehat{x}(\e_1)$ satisfying $\e_1 \le \frac{1}{2}$ and $\widehat{x}(\e_1) \in B_{\frac{R}{2}}(0)$. After setting $\delta_2 = \frac{1}{2}$, we also find $\e_2 \in A$ and $\widehat{x}(\e_2)$ satisfying $\e_2 \le \min\left(\e_1, \frac{1}{2^2}\right)$ and $\widehat{x}_2(\e_2) \in B_{\frac{R}{2^2}}$.
In this way, we can obtain a sequence $(\delta_k,\e_k,x_k = \widehat{x}(\e_k,\delta_k))$ satisfying
\begin{enumerate}
\item
$\{ \e_k \}$ is a subsequence of $A$,
\item
$\delta_k \ra 0$, $\e_k \ra 0$, and $x_k \ra 0$ as $k \ra \infty$, and
\item
$u_{\e_k}(x) \le P_k(x) \text{ and } u_{\e_k}(x_k) = P_k(x_k)$ in $B_R(0)$
where $P_k(x) = P^{\delta_k}(x) - P^{\delta_k}(x_k) + u_{\e_k} (x_k) $.
\end{enumerate}

Let $Q_k(x) = P_k(x) - \displaystyle\frac{K(\eta)}{2} \left( |x|^2 - \frac{R^2}{2} \right)$ for a given $K \le 1$. Then, from $|x_k| \le \displaystyle\frac{R}{2}$, $Q_k$ satisfies,
\begin{equation*}
Q_k(x_k) = P_k(x_k) -\frac{K}{2} \left( |x_k|^2 - \frac{R^2}{2} \right) \geq  u_{\e_k}(x_k) + \displaystyle\frac{K R^2}{8}
\end{equation*}
and
\begin{equation*}
Q_k(x) = P_k(x) - \frac{K}{2} \left( R^2 - \frac{R^2}{2} \right) \le u_{\e_k}(x) - \frac{K R^2}{4}
\end{equation*}
on $\partial B_R(0)$.

Let $P(x) = \displaystyle\frac{1}{2}x^t M x + px + u_*(0)$, $M_k = D^2 Q_k = M - 2\delta_k I - K(\eta) I$ and $\xi_k = \xi_k(x) = DQ_k(x) = Mx+p -2\delta_kx - K(\eta) x = \xi(x) -2\delta_kx - K(\eta) x $. We note that $\xi_k(0) =\xi(0)=p$.
Now, let us define the first and second corrector as follow.
\begin{equation*} \begin{cases}
v_k \left( \displaystyle\frac{x}{\e_k},\xi_k \right) = V \left( \displaystyle\frac{x}{\e_k} \right) \xi_k(x) \\
w_k \left( \displaystyle\frac{x}{\e_k} \right) = w_{\e_k} \left( \displaystyle\frac{x}{\e_k} ; M_k, Q_k(x_0),x_0 \right)
\end{cases} \end{equation*}
And, define
\begin{equation*}
\widetilde{Q}_k(x) = Q_k(x) +\e_k v_k \left( \frac{x}{\e_k}, \xi_k \right) + \e_k^2 \widehat{w}_k \left( \frac{x}{\e_k} \right)
\end{equation*}
where
$\widehat{w}_k \left( \displaystyle\frac{x}{\e_k} \right) = w_k \left( \displaystyle\frac{x}{\e_k} \right) - \min_{y \in \R^n \setminus T^{\mathfrak a}} w_k \left( y \right)$.

We will show that $\widetilde{Q}_k(x)$ is a sub-solution in a ball $B_R(0)$ if we choose $R$ and $K$ properly. First, let us check the boundary condition on $B_R(0) \cap \Omega_\e$.
\begin{equation*} \begin{aligned}
b^i \left( \frac{x}{\e_k} \right) D_i \widetilde{Q}_k(x) &=  b^i \left( \frac{x}{\e_k} \right) \left( \xi_k^i(x) + D_i v^l \left( \frac{x}{\e_k} \right) \xi_k^l(x) + \e_k M^{li} \left( v^l \left( \frac{x}{\e_k} \right) + D_i w_k \left( \frac{x}{\e_k} \right)\right) \right) \\
&= -\e_k \e_k^2 w_k \left( \frac{x}{\e_k} \right) \\
&= -\e_k \left( \overline{L}(M_k,Q_k(x_0),x_0) + o(\e_k) \right) \\
\end{aligned} \end{equation*}

From the continuity of $\overline{L}$, $\overline{L}(M_k,Q_k(x_0),x_0)$ converges to $\overline{L}\left( M,\displaystyle\frac{KR^2}{4},0 \right)$ since $M_k \ra M$ and $Q_k(x_0) = \displaystyle\frac{KR^2}{4}$. And $\overline{L}\left( M,\displaystyle\frac{KR^2}{4},0 \right)$ is positive if we choose $K$ and $R$ small enough since $\overline{L}(M,0,0)$ is positive. Hence we have $b^i \left( \frac{x}{\e_k} \right) D_i \widetilde{Q}_k(x) \le 0$ for sufficiently large $k$.

Now we are going to apply $\widetilde{Q}_k$ to our main equation \eqref{eqn-main-1}. From the calculation, we have
\begin{equation*}
D^2 \widetilde{Q}_k(x) M_k + \displaystyle\frac{1}{\e_k} D^2 v^l \left( \displaystyle\frac{x}{\e_k} \right) + Z\left( \displaystyle\frac{x}{\e_k}, M_k \right) + D^2 w_k \left( \displaystyle\frac{x}{\e_k} \right).
\end{equation*}

Apply it to the equation \eqref{eqn-main-1}. Then we have
\begin{equation*} \begin{aligned}
a_{ij} &\left( \displaystyle\frac{x}{\e_k} \right) D^2 \widetilde{Q}_k + c\left( \widetilde{Q}_k(x),x,\displaystyle\frac{x}{\e_k} \right) - f\left( x,\displaystyle\frac{x}{\e_k} \right) \\
&= \e_k^2 w_k \left( \displaystyle\frac{x}{\e_k} \right) -\left( c\left( Q_k(x_0),x_0,\displaystyle\frac{x}{\e_k} \right) - f\left( x_0,\displaystyle\frac{x}{\e_k} \right) \right)+ c\left( \widetilde{Q}_k(x),x,\displaystyle\frac{x}{\e_k} \right) - f\left( x,\displaystyle\frac{x}{\e_k} \right).
\end{aligned} \end{equation*}

From the definition of $Q_k(x)$,
\begin{equation*} \begin{aligned}
\left| Q_k(x) - Q_k(x_0) \right| &\le \left| P(x) -P(x_0) \right| + \delta_k |x|^2 + \displaystyle\frac{K}{2} |x|^2 \\
&\le \| M \| R^2 + |p| R + \delta_k R^2 + \displaystyle\frac{K}{2} R^2
\end{aligned} \end{equation*}
whenever $x \in B_R(x_0)$. Hence we can make $\left| Q_k(x) - Q_k(x_0) \right| $ small by choosing $R(\eta)$ small enough.

And, from lemma \ref{lem-cor1-osc-uni} and \ref{lem-cor2-osc},
\begin{equation*} \begin{aligned}
\left| \widetilde{Q}_k(x) - Q_k(x) \right| &\le \e_k |v_k|_\infty + \e_k^2 |\widehat{w}_k|_\infty \\
&\le C \e_k
\end{aligned} \end{equation*}
for some constant $C$ which is uniform on $\e$ and $k$. Hence it is smaller than $\eta$ if $k$ is large enough.

From those two calculations, $\left| Q_k(x) - Q_k(x) \right|$ satisfies the following by small $R$ and large $k$.
\begin{equation*}
\left| c\left( \widetilde{Q}_k(x),x,\displaystyle\frac{x}{\e_k} \right) - c\left( Q_k(x_0),x,\displaystyle\frac{x}{\e_k} \right) \right| \le \eta
\end{equation*}

Similarly, because of the continuity of $c$ and $f$ with $x$ variable,
\begin{equation*} \begin{aligned}
\left| c\left( Q_k(x_0),x,\displaystyle\frac{x}{\e_k} \right) - c\left( Q_k(x_0),x_0,\displaystyle\frac{x}{\e_k} \right) \right| &\le \eta \\
\left| f(x,y) - f(x_0,y) \right| &\le \eta
\end{aligned} \end{equation*}

Hence, from the continuity of $\overline{L}$, we have
\begin{equation*} \begin{aligned}
a_{ij} \left( \displaystyle\frac{x}{\e_k} \right) D^2 \widetilde{Q}_k + c\left( \widetilde{Q}_k(x),x,\displaystyle\frac{x}{\e_k} \right) - f\left( x,\displaystyle\frac{x}{\e_k} \right) &\ge \overline{L}(M_k,Q_k(x_0),x_0)  -3\eta + o(\e_k) \\
&\ge \overline{L}\left( M,\displaystyle\frac{KR^2}{4},0 \right) -4\eta  \\
&\ge \eta + o(\e_k) \\
&\ge 0
\end{aligned} \end{equation*}
for large $k$.

In summary, $\widetilde{Q}_k(x)$ is a sub-solution of equation \eqref{eqn-main-1} in $B_R(x_0)$ for large $k$ and  $\widetilde{Q}_k(x) \le P_k(x) \le u_{\e_k}(x)$ on $\partial B_R \setminus T_\e$. Hence, from the comparison, we have
\begin{equation*}
\widetilde{Q}_k(x) \le u_{\e_k} (x)
\end{equation*}

in $B_R \setminus T_\e$. Substitute $x_k$ instead of $x$ to above equation and take limit on both side. Then, we have
\begin{equation*}
u_*(0) + \displaystyle\frac{KR^2}{4} \le u_*(0).
\end{equation*}
That is a contradiction. So, $\overline{L}(D^2P(x_0),P(x_0),x_0)$ is nonpositive and hence $u_*$ is a super-solution at any point in $\Omega$.  By using similar argument, we can show $u^*$ is a sub-solution of the equation \eqref{eqn-rst-1} in $\Omega$. 
From the second assumption, $u^* =u_*$ on $\partial \Omega$. So, $u_* \ge u^*$ from the comparison principle. Finally, since $u_* \le u^*$ because of the definition of $u_*$ and $u^*$, we conclude
\begin{equation*}
u_* = u^* \text{ in } \Omega.
\end{equation*}
\end{proof}

\subsection{Construction of barriers when $\Omega$ is convex}
In this section, we are going to construct a barrier to show $u^* =u_*$ on the $\partial \Omega$. At first, we are going to prove that the condition (2) in \ref{thm-main-1} holds if $\Omega$ is convex. 
\begin{lem} \label{lem-bar}
Let $u_\e$ be the solution of the equation \eqref{eqn-main-1}. Assume that the equation \eqref{eqn-main-1} satisfies the condition I and compatibility condition and the size of halls $a$ is less than or equal to $ {\mathfrak a}_0$ where ${\mathfrak a}_0$ is same with the constant in theorem \ref{thm-cor2-elliptic}. Assume also that $\Omega $ is convex. Then $u_\e$ is bounded uniformly on $\e$. Moreover, for any given $x_0 \in \partial
 \Omega$, we can find a barrier functions $h^+_\e$ and $h^-_\e$ which is bounded uniformly on $\e$ and satisfying
\begin{equation} \label{eqn-bdd-h} \begin{cases}
h^-_\e(x) \le u_\e(x) \le h^+_\e(x) \text{ in } \Omega \\
\left| h^+_\e(x_0) - h^-_\e(x_0) \right| \le C \e \text{ for some } C \text{ which is uniform on } \e.
\end{cases} \end{equation}
\end{lem}
\begin{proof}
For the simplicity we assume that $x_0=0$ and $\vp(x_0) = \vp(0) = 0$ and $\Omega$ is contained in a half space $x^n \ge 0$. We also assume that $|f(x,y)|_\infty \le 1$ and $\| \vp\|_{C^2(\overline{\Omega})} \le 1$ since the general case can be obtained by scaling.

Let $v^i$ be the solution of equation \eqref{eqn-cor1} when $\xi =e^i$ with $\min_{\R^n \setminus T^{\mathfrak a}} v^i =0$. Let $V$ be a vector whose i-th component is $v^i(y)$.

Since $\| \vp \|_{C^2(\overline{\Omega})} \le 1$, $\vp(x) \le \vp(x_0) + D\vp(x_0) \cdot x + \displaystyle\frac{1}{2} |x|^2$. Choose $R = diam(\Omega)$ and define $P(x) = D\vp(0) \cdot x + \displaystyle\frac{1}{2} |x|^2 + \displaystyle\frac{K}{2} \left( R^2 - |x^n - R|^2 \right)$ for some $K$. Define also $R_0 = R + \displaystyle\frac{R^2}{2}$,  and $\xi(x) = DP(x) = D\vp(0) + x - K\left(x^n e^n- R\right)$ where $e^n$ is the n-th stanadard unit vector defined in $\R^n$.
We finally define $P_\e(x) = P(x) + \e V\left(\displaystyle\frac{x}{\e}\right) \xi(x) + \e^2 \widehat{w}_\e \left(\displaystyle\frac{x}{\e};I -KE_n,-R_0,0\right) + C_1 \e $ where $C_1=C_1(K)$ is a constant satisfying $\| V \|_{L^\infty(\R^n \setminus T^{\mathfrak a})} |\xi(x)|_\infty \le C_1$ and $E_n$ is a matrix with defined by $E_n = e^n (e^n)^t$. By adding $C_1\e$, $P_\e \ge - \sup_{x \in \Omega} \left( D\vp(0) \cdot x + \displaystyle\frac{1}{2}|x|^2 \right) \ge -\left( R + \displaystyle\frac{R^2}{2} \right) = - R_0$ in $\Omega$ and $P_\e \ge \vp$ on $\partial \Omega$.

Then, from the calculation, we have
\begin{equation*} \begin{aligned}
DP_\e(x) &= \xi(x) + DV \left(\frac{x}{\e} \right) \xi(x) + \e (I-KE_n) V\left(\frac{x}{\e} \right) + \e Dw_\e \left(\frac{x}{\e} \right) \\
D^2P_\e(x) &= (I-KE_n) + \displaystyle\frac{1}{\e}DV \left(\frac{x}{\e} \right) \xi(x) + (I-KE_n) DV \left(\frac{x}{\e} \right) + D^2 w_\e\left(\frac{x}{\e} \right)
\end{aligned} \end{equation*}
Now, apply it to our main equation, then we have
\begin{equation} \begin{aligned}
L &\left( D^2 P_\e(x), P_\e(x) ,x,\frac{x}{\e} \right) = a_{ij} \left(\frac{x}{\e} \right)D_{ij}P_\e(x) + c \left(P_\e(x),x,\frac{x}{\e} \right) \\
&= a_{ij} \left(\frac{x}{\e} \right) \left( (I-KE_n) + \displaystyle\frac{1}{\e}DV \left(\frac{x}{\e} \right) \xi(x) + (I-KE_n) DV \left(\frac{x}{\e} \right) + D^2 w_\e(\frac{x}{\e}) \right)_{ij} + c\left(P_\e(x),x,\frac{x}{\e}\right) \\
&\le a_{ij} \left(\frac{x}{\e} \right) \left( (I-KE_n) + (I-KE_n) DV \left(\frac{x}{\e} \right) + D^2 w_\e \left(\frac{x}{\e} \right) \right)_{ij} + c \left(-R_0,x,\frac{x}{\e} \right) \\
&=\e^2 w_\e(y;I-KE_n,-R_0,0) + \left( c \left(-R_0,x,\frac{x}{\e} \right) - c \left(-R_0,0,\frac{x}{\e} \right) \right) + f\left(0,\displaystyle\frac{x}{\e} \right).
\end{aligned} \end{equation}
By the theorem \ref{thm-cor2-elliptic},
\begin{equation} \begin{aligned}
\e^2 w_\e(y;I-KE_n,-R_0,0) &\le \overline{L}(I-KE_n,-R_0,0) + o(\e) \\
&\le  n\Lambda -\displaystyle\frac{\overline{\lambda} K}{2} + \overline{L}(0,-R_0,0)
\end{aligned} \end{equation}
if $\e$ is small enough.
From the definition, we can easily deduce that $ | \overline{L}(0,-R_0,0) | \le \| f(x,y) \|_{L^\infty(\Omega \times (\R^n \setminus T^{\mathfrak a}))} + \| c(-R_0,x,y) \|_{L^\infty(\Omega \times (\R^n \setminus T^{\mathfrak a}))} $. And hence if we choose
\begin{equation*}
K = \displaystyle\frac{2n \Lambda +6 \| f(x,y) \|_{L^\infty(\Omega \times (\R^n \setminus T^{\mathfrak a}))} + 6\| c(-R_0,x,y) \|_{L^\infty(\Omega \times (\R^n \setminus T^{\mathfrak a}))} }{\overline{\lambda}},
\end{equation*}
then,
\begin{equation*}
L \left( D^2 P_\e(x), P_\e(x) ,x,\frac{x}{\e} \right) \le f \left(x,\frac{x}{\e} \right)
\end{equation*}
for every $x \in \Omega_\e$.
And, at the boundary,
\begin{equation*} \begin{aligned}
b^i \left(\frac{x}{\e} \right) D_iP_\e(x) &= b^i\left(\frac{x}{\e} \right) \left( \xi(x) + DV\left(\frac{x}{\e}\right) \xi(x) + \e (I-KE_n) V\left(\frac{x}{\e}\right) + \e Dw_\e\left(\frac{x}{\e}\right) \right)_i \\
&= \e b^i\left(\frac{x}{\e}\right) \left( (I- KE_n) V \left(\frac{x}{\e}\right) + Dw_\e\left(\frac{x}{\e}\right) \right)_i \\
&= -\e^2 w_\e\left(\frac{x}{\e}\right) \\
&\ge 0
\end{aligned} \end{equation*}

Finally, since $P_\e(x) \ge \vp(x)$ on $\partial \Omega$, $P_\e$ is a super-solution of \eqref{eqn-main-1} and hence $u_\e \le P_\e$ on $\Omega_\e$.
And the uniform boundedness comes from the fact that $\e \le1$, and $C_1$ and $\widehat{w}_\e$ is bounded uniformly on $\e$.

Let $h^+_\e = P_\e$. Then, $h^+_\e(x) \ge u_\e(x)$ in $\Omega$ and $h^+_\e(x_0) = h^+_\e(0) \le P(0) + C\e = \vp(0) + C\e$. By using similar argument, we can construct $h^-_\e$ having properties  $h^-_\e(x) \le u_\e(x)$ in $\Omega$ and $h^-_\e(x_0) \ge \vp(0) - C\e$. And such $h^+_\e$ and $h^-_\e$ satisfy \eqref{eqn-bdd-h}.
\end{proof}

\begin{cor}
Let $u_\e$ be the solution of equation \eqref{eqn-main-1} and assume all the conditions in lemma \ref{lem-bar}. Then, for any given boundary point $x_0$,
\begin{equation*}
u^*(x_0) = \vp(x_0) = u_*(x_0).
\end{equation*}
\end{cor}
\begin{proof}
From lemma \ref{lem-bar}, there is a function $h^+_\e$ satisfying $0 \le h^+_\e(x) -u_\e(x) \le C \e$ in $\Omega$ for some constant $C$ which is uniform on $\e$. And from the definition of $h^+_\e$ in lemma \ref{lem-bar}, $h^+_\e$ converges to $P(x)$ uniformly.  So, $u^*(x) \le P(x) $ and hence $u^*(x_0) \le P(x_0) = \vp(x_0)$.
Similarly, we can show $u_*(x_0) \ge  \vp(x_0)$. Finally, since $u_*(x_0) \le u^*(x_0)$, $u_*(x_0) = \vp(x_0) = u^*(x_0)$.
\end{proof}

\subsection{Construction of barriers for the non-convex domain}
In this section, we construct a barrier to show $u_* = u^*$ on $\partial \Omega$ for non-convex domain $\Omega$. Throughout this section, we don't assume that $\Omega$ is convex. Instead, we assume that $\Omega$ has a exterior sphere condition. In the other words, for given any $x_0$ on $\partial \Omega$, there is a ball $B_r(x_1)\subset \R^n\backslash \Omega$ satisfying $\{x_0\} \subset \overline{B_r(x_1)} \cap \overline{\Omega}$. Let us also define the set of functions:
\begin{equation*}
\mathcal{A} = \left\{ f(x,y) \in \mathcal{C} (\overline{\Omega} \times \R^n) : \left| f(x_1,y) - f(x_2,y) \right| \le \sigma(|x_1 -x_2| ) \text{ for all } y \in \R^n \setminus T^{\mathfrak a} \right\}
\end{equation*}
for some nondecreasing function $\sigma$ satisfying  $\sigma(0+) =0$. We note that $f$ satisfying  (Condition I) is in $\mathcal{A}$. So, for given any $f(x,y)$ in $\mathcal{A}$, we may find $-\overline{L}(0,0,x)$. We define $\overline{f} = -\overline{L}(0,0,x)$. In other words, $\overline{f}(x)$ is the limit of $-\e^2 w_\e(y)$ and $w_\e$ is the solution of the following equation for given $f(x,y)$:
\begin{equation} \label{eqn-cor2-f}\begin{cases}
-\e^2 w_\e(y) + a_{ij}(y) D_{ij}w_\e(y)= f(x,y) &\text{ in } \R^n \setminus T^{\mathfrak a} \\
b^i(y)  D_i w_\e(y) +\e^2 w_\e(y) = 0 &\text{ on } \partial T^{\mathfrak a}
\end{cases} \end{equation}

\begin{lem}
Suppose that $f(x,y) = 1$ identically. Then, $\overline{f}$ is a positive constant which depends only on $n$, $\lambda$, $\Lambda$, and the size of holls $a$.
\end{lem}
\begin{figure} \label{fig-q}[ht]
\includegraphics[width=60mm]{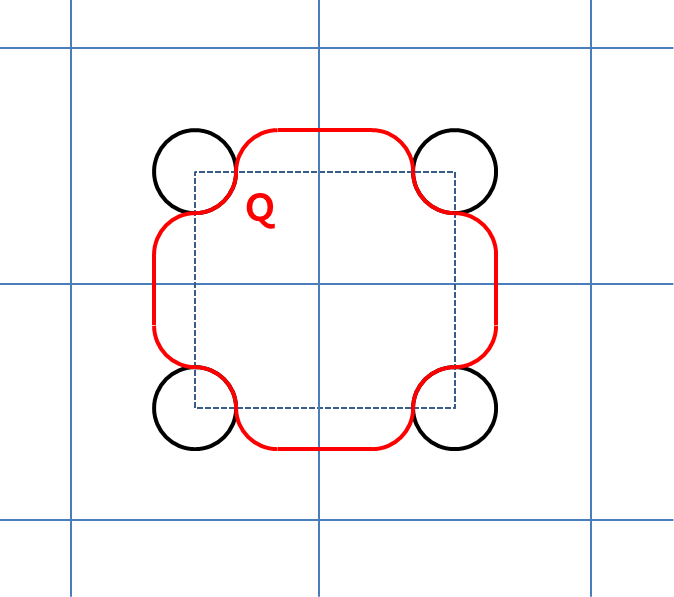}
\caption{The shape of Q}
\end{figure}
\begin{proof}
It is enough to prove $\overline{-1}$ is a negative constant. 
Let $Q$ be a subset of $\R^n \setminus T^{\mathfrak a}$ having the shape in the picture and $y_0 = \left( \displaystyle\frac{1}{2}, \frac{1}{2}, \cdots, \frac{1}{2} \right)$ be a point in $Q$. We can choose $Q$ that has a $\mathcal{C}^2$ boundary. Let $h_1$ be the solution of the following equation:
\begin{equation*} \begin{cases}
a_{ij}(y) D_{ij} h_1(y) = 0 &\text{ in } Q\setminus B_{\frac{1}{3}} (y_0) \\
h_1(y) = 0 &\text{ on } \partial Q \\
h_1(y) = 1 &\text{ on } \partial B_{\frac{1}{3}} (y_0)
\end{cases} \end{equation*}

Then, from the regularity theory of viscosity solution in \cite{LT} and Hopf's lemma(lemma 3.4  in \cite{GT}), $h_1$ satisfies the following:
\begin{equation*} \begin{cases}
h_1 \text{ is in  } \mathcal{C}^2 \left( \overline{ Q \setminus B_{\frac{1}{3}} (y_0) } \right) \\
0 \le h_1 \le 1 \text{ in }  \overline{ Q \setminus B_{\frac{1}{3}} (y_0) } \\
D h_1 = |D h_1 | \nu \text{ and } |D h_1 | \le c_2 \text{ on } \partial B_{\frac{1}{3}} (y_0)\\
D h_1 = -|D h_1 | \nu \text{ and } |D h_1 | \ge c_1 \text{ on } \partial Q\\ 
\end{cases} \end{equation*}
where $c_1$ and $c_2$ are positive constants depending only on $n$, $\lambda$, $\Lambda$ and $\mu$. 
Now extend $h_1 = 0 $ outside of $Q$ and $h_1 = \displaystyle\frac{K}{2} \left( \left| \displaystyle\frac{1}{3} \right|^2 - |y-y_0|^2 \right) $ in $B_{\frac{1}{3}} (y_0)$. Hence, if we choose $ K \ge 3(c_2 +1)$, then $h_1$ satisfies 
\begin{equation*} \begin{cases}
a_{ij}(y) D_{ij} h_1 \ge -n \Lambda K &\text{ in } \R^n \setminus T^{\mathfrak a} \\
b^i(y) D_i h_1 \le - \mu c_1 &\text{ on } \partial T^{\mathfrak a} \cap \partial Q
\end{cases} \end{equation*}
(in the viscosity sense) since the  sharp edge on $\partial Q \cup \partial B_{\frac{1}{3}} (y_0)$ does not allow any second polynomial to touch $h_1$ from above on $\partial Q \cup \partial B_{\frac{1}{3}} (y_0)$. 

So, the function $h_\e = \displaystyle\frac{1}{2n\Lambda K} \sup_{k \in \Z^n} h_1(y-k) + \displaystyle\frac{1}{\e^2}  \displaystyle\frac{\mu c_1}{4n\Lambda K} $ is a sub-solution of equation \eqref{eqn-cor2-f} for small $\e$.  Hence the solution $w_\e$ of  \eqref{eqn-cor2-f} with $f(x,y)=-1$ should be larger than $h_\e$ by lemma \eqref{eqn-cor2-f}.  Now  we get the conclusion since $-\e^2 h_\e \le -\displaystyle\frac{\mu c_1}{4n\Lambda K}  < 0$.

\end{proof}

\begin{lem} \label{lem-bar-inv}
For given any $\phi \in \mathcal{C} (\overline{\Omega})$ and $f \in \mathcal{A}$, there is a function $g(x) \in \mathcal{C} (\overline{\Omega})$ satisfying
\begin{equation*}
\overline{f+g} = \phi(x)
\end{equation*}
\end{lem}
\begin{proof}
By above lemma, there is a positive constant $c_0$ such that $\overline{1} =c_0$. From the definition of $\overline{f}$, if $f$ is independent on $y$, then $\overline{f}(x) = \overline{f(x)}$ for each fixed $x \in \overline{\Omega}$. From this and the linearity of the equation \eqref{eqn-cor2-f}, $\overline{f}(x) = c_0 f(x)$ for all $f \in \mathcal{C}(\overline{\Omega})$.  
Now let us define $g(x)$ as follow:
\begin{equation*}
g(x) = \displaystyle\frac{\phi(x) - \overline{f}(x)}{c_0}.
\end{equation*}
Then, $g(x)$ is continuous since $\overline{f}(x)$ and $\phi(x)$ is continuous and,
\begin{equation*}
\overline{f(x,y) + g(x) } = \overline{f} + \overline{g} = \overline{f} + \left( \phi(x) - \overline{f} \right) = \phi(x) .
\end{equation*} 

\end{proof}

Let $x_0$ be a point on $\partial \Omega$. We assume that $\Omega$ has exterior sphere condition. So, there exists a ball $B_r(x_1)$ such that $\overline{B_r(x_1)} \cap \overline{\Omega} = \{x_0 \}$.
We may assume that $x_1=0$ without any loss of generality. Since $\Omega$ is bounded, we can find a large ball $\mathcal{O}$ containing $\Omega$ and $B_r(0)$. Extend $c(r,x,y)$ and $f(x,y)$ when $x \in \mathcal{O}$ and hence $\widetilde{L}(M,r,x) = \overline{L}(M,r,x) - \overline{f}(x)$ also defined on $\mathcal{O}$.

\begin{lem} \label{lem-bar-h}
Assume that $\widetilde{L}$ is uniformly elliptic. Then, there is a function $h \in \mathcal{C}^2(\overline{\mathcal{O}})$ satisfying the following:
\begin{equation*} \begin{cases}
\widetilde{L}(D^2 h,h,x) \le -\left( \| f\|_\infty +1 \right) \text{ on } \Omega \\
h \ge \vp \text{ on } \partial \Omega \\
h(x_0) = \vp(x_0). \\
\end{cases}\end{equation*}
\end{lem}
\begin{proof}
We may assume that $\vp$ is defined in $\mathcal{O}$ since it can be easily extended. Let $h(x) = \vp(x) + \left( - \displaystyle\frac{K}{|x|^\alpha} +  \displaystyle\frac{K}{r^\alpha} \right) $ outside the ball $B_r(0)$. Then, by defining $h(x)$ properly inside the ball $B_r$, we may assume that $h$ is in $\mathcal{C}^2$ in $\mathcal{O}$. And, from the definition, the second and third statement are true for all $\alpha >0$ and $K>0$. So, we get the conclusion by choosing $\alpha$ and $K$ large enough.
\end{proof}

Let us define the function $\widetilde{g} = \widetilde{L}(D^2 h,h,x) $ on $\mathcal{O}$. Then, from proposition \ref{prop-cor2-conti} and lemma \ref{lem-bar-h}, $\widetilde{g}$ is continuous and, by lemma \ref{lem-bar-inv}, there is a continuous function $g(x)$ satisfying the following:
\begin{equation*}
\overline{f(x,y) + g(x)} = \widetilde{g}(x) 
\end{equation*}
for any given $f \in \mathcal{A}$.

\begin{lem} \label{lem-bar-g}
Let $g$ be the function defined as above for given $f \in \mathcal{A}$. Then, $ g(x) \le 0$.
\end{lem}
\begin{proof}
From the definition and the comparison, $|\widetilde{f}| \le |f|_\infty$. And from the construction of $h$, $\widetilde{g} \le -(|f|_\infty +1 )$
By combining those two inequalities, we have $g(x) = \displaystyle\frac{ \widetilde{g}(x) - \overline{f}(x) }{c_0} \le 0$.
\end{proof}

Now consider the equation
\begin{equation*} \begin{cases}
L(D^2 v_\e, v_\e, x) = f\left(x,\frac{x}{\e} \right) + g(x) &\text{ in } \mathcal{O} \\
b^i\left(\frac{x}{\e}\right) D_iv_\e(x) = 0 &\text{ on } \partial T_\e \cap \mathcal{O}\\
v_\e = h(x) &\text{ on } \partial \mathcal{O} \setminus T_\e
\end{cases} \end{equation*}
where $h$ is same in lemma \ref{lem-bar-h} and $g$ is same in lemma \ref{lem-bar-g}.

Since $\mathcal{O}$ is convex, if we assume all the conditions in the theorem \ref{thm-main-1} except (2), then $v_\e$ converges to $\widetilde{v}$ uniformly and $\widetilde{v}$ satisfies the equation
\begin{equation*} \begin{cases}
\widetilde{L}(D^2 \widetilde{v},\widetilde{v},x) = \widetilde{g} \text{ on } \mathcal{O} \\
\widetilde{v} = h \text{ on } \partial \mathcal{O}.
\end{cases} \end{equation*}

And hence, $h(x) = \widetilde{v}$ because of the comparison. Let us define $\tau(\e) = \sup_{x \in \mathcal{O}} | v_\e(x)  - h(x) |$. Then, from the uniform convergence, $\tau(\e) \ra 0$ as $\e \ra 0$. And, from the definition of $h$,
\begin{equation*}
u_\e(x) \le h(x) \le v_\e(x) + \tau(\e) \text{ on } \partial \Omega.
\end{equation*}
Since $v_\e$ is a super-solution in $\Omega$, we have the following by comparison:
\begin{equation*}
u_\e(x) \le  v_\e(x) + \tau(\e) \le h(x) + 2 \tau(\e) \text{ for } x \in \Omega.
\end{equation*}
By taking $*$ on both side, we can conclude the following
\begin{equation*}
u^*(x_0) \le h^*(x_0) = h(x_0) = \vp(x_0).
\end{equation*}
Similarly, we can show $u_*(x_0) \ge \vp(x_0)$. Hence we get the following:
\begin{lem}
Let $u_\e$ is the solution of equation \eqref{eqn-main-1}. Assume that the equation \eqref{eqn-main-1} satisfies the condition I and the compatibility condition. Suppose also that $\widetilde{L}$ is uniformly elliptic. Then, $u*(x) = u_*(x)$ on $\partial \Omega$ if $\Omega$ satisfies an exterior sphere condition. 
\end{lem}

\begin{proof}[\bf{Proof of corollary \ref{cor-main-1} and corollary \ref{cor-main-2}}] \item
It follows immediately from theorem \ref{thm-main-1} and lemma above. 
\end{proof}

\section{Discrete Gradient Estimate}\label{sec-Discrete}
\subsection{}
In this section,we develop the following uniform estimates. Those two estimate tell us about the shape of $u_\e$. It turns out that $u_\e$ is almost Lipschitz continuous with an error of $\e$ order. Let us consider the following equation:
\begin{equation} \label{eqn-disc}
\begin{cases}
a_{ij}\left( \displaystyle\frac{x}{\e} \right) D_{ij}u_\e(x) + c\left(u_\e,\displaystyle\frac{x}{\e} \right) = f\left(x,\frac{x}{\e}\right) &\text{ in } \Omega \setminus T_\e \\
b^i\left(\displaystyle\frac{x}{\e} \right) D_iu_\e(x) = 0 &\text{ on } \Omega \cap \partial T_\e \\
u_\e = \vp &\text{ on } \partial \Omega \setminus T_\e
\end{cases}
\end{equation}
We note that the function $c$ is independent on $x$ variable. It is only difference between above equation and the main equation \eqref{eqn-main-1}. Throughout this section, we assume that the function $c(r,y)$ and $f(x,y)$ is differentiable with respect to $r$ and $x$ variable respectively. Additionally, we assume that $\Omega$ is convex and all the assumptions in chapter 1.
\begin{lem}[Discrete Gradient Estimate]\label{lem-disc-grad}
Let $u_\e$ be the solution of above equation and, the size of holes $a$ is smaller than ${\mathfrak a}_0$ where ${\mathfrak a}_0$ is same in theorem \ref{thm-cor2-elliptic}. Then, for given direction $e=e_k$, $|\La_{e}^{\e}u_{\e}(x)|$ is bounded uniformly on $\e$. That is, there exist $C=C(n,\lambda,\Lambda,\mu,\|f\|_\infty)$ and
\begin{equation}\label{eq-disc-grad}
|\La_{e}^{\e}u_{\e}(x)|=\displaystyle \frac{|u_{\e}(x+\e e)-u(x)|}{\e}\leq C
\end{equation}
for every $x \in \Omega_\e \cap (\e e + \Omega_\e)$.
\end{lem}
\begin{proof}
Let $\widetilde{\Omega}^\e = \Omega \bigcap (\e e + \Omega)$, $\widetilde{\Omega}_\e = \widetilde{\Omega}^\e \setminus \overline{T}_\e$ and $U = \displaystyle \frac{u_\e(x+ \e e) -u_\e(x)}{\e}$ for given $e$ where $e=e^i$ is a i-th standard unit vector in $\R^n$. 
Then, since $u_\e(x)$ and $u_\e(x+\e e)$ are solutions, we have
\begin{equation*} \begin{aligned}
a_{ij}\left(\displaystyle\frac{x+ \e e}{\e}\right) D_{ij}u_\e(x+\e e) + c\left(u_\e(x+ \e e), \displaystyle\frac{x+ \e e}{\e}\right) &= f\left(x + \e e, \frac{x+ \e e}{\e}\right) \\
a_{ij}\left(\displaystyle\frac{x}{\e}\right) D_{ij}u_\e(x) + c(u_\e(x), \frac{x}{\e}) &= f\left(x, \displaystyle\frac{x}{\e}\right). \\
\end{aligned} \end{equation*}

Since $a_{ij}\left(\displaystyle\frac{x+ \e e}{\e}\right)=a_{ij}\left(\displaystyle\frac{x}{\e}\right)$, $c\left(u_\e(x+ \e e), \displaystyle\frac{x+ \e e}{\e}\right) = c\left(u_\e(x+ \e e), \displaystyle\frac{x}{\e}\right)$ and $f\left(x + \e e, \displaystyle\frac{x+ \e e}{\e}\right) = f\left(x + \e e, \displaystyle\frac{x}{\e}\right)$, we have
\begin{equation*}
a_{ij}\left(\displaystyle\frac{x}{\e}\right) D_{ij} U(x) + \displaystyle\frac{c\left(u_\e(x + \e e),\frac{x}{\e}\right) - c(u_\e(x),\frac{x}{\e})}{\e} = \displaystyle\frac{f\left(x +\e e,\frac{x}{\e}\right) - f\left(x,\frac{x}{\e}\right)}{\e}.
\end{equation*}
Because $c$ is differentiable in $r$ variable and $f$ is differentiable in $x$ variables, by the mean value theorem, we can find $r^* = r^*(x,\e,e)$ and $x^*=x^*(x,\e,e)$ which satisfies
\begin{equation*}
a_{ij}\left(\frac{x}{\e}\right) D_{ij} U(x) + c_r\left(r^*,\frac{x}{\e}\right) U = D_e f\left(x^*,\frac{x}{\e}\right)
\end{equation*}
where $x \in \widetilde{\Omega}_\e$.
And from the boundary condition of $u_\e(x)$ and $u_\e(x + \e e)$, $b^i\left(\displaystyle\frac{x}{\e}\right)D_i U(x) = 0$ on $\partial T_\e \bigcap \widetilde{\Omega}^\e$.

Now we are going to prove the boundedness of $U$ on $\partial \widetilde{\Omega}^\e \setminus \overline{T}_\e$. Let $x$ be a point in $\widetilde{\Omega}^\e \setminus \overline{T}_\e$. Since $x \in \widetilde{\Omega}^\e \setminus \overline{T}_\e$, $x \in \partial \Omega$ or $x + \e e \in \partial \Omega$. We assume $x \in \partial \Omega$ since the other case is similar. Let $h^+_\e$ and $h^-_\e$ be functions satisfying \eqref{eqn-bdd-h}. Then, we have
\begin{equation*} \begin{aligned}
U(x) &= \displaystyle\frac{u_\e(x+\e e) -u_\e(x)}{\e} \\
&\le \displaystyle\frac{h^+_\e(x+\e e) -h^-_\e(x)}{\e} \\
&\le \displaystyle\frac{h^+_\e(x+\e e) -h^+_\e(x)}{\e} + \displaystyle\frac{h^+_\e(x, e) -h^-_\e(x)}{\e}  \\
&= A_1 + A_2
\end{aligned} \end{equation*}

By \eqref{eqn-bdd-h}, $A_2$ is bounded and, by the definition of $h^+_\e$, $A_1$ is bounded. So, $U(x)$ is bounded above. And it also bounded below by using similar argument.

In summary, $U$ satisfies
\begin{equation*} \begin{cases}
a_{ij}\left(\displaystyle\frac{x}{\e}\right) D_{ij} U(x) + c_r\left(r^*,\displaystyle\frac{x}{\e}\right) U = D_e f\left(x^*,\displaystyle\frac{x}{\e}\right) &\text{ in } \widetilde{\Omega}_\e \\
b^i\left(\displaystyle\frac{x}{\e}\right) D_i U(x) = 0 &\text{ on } \partial T_\e \bigcap \widetilde{\Omega}^\e \\
|U| \text{ is bounded } &\text{ on } \partial \widetilde{\Omega}^\e \setminus \overline{T}_\e.
\end{cases} \end{equation*}

It can be shown from the comparison if there is a super-solution which is bounded uniformly on $\e$, then we are done. Let $h^+ = \sup_{x \in \partial \widetilde{\Omega}^\e \setminus T_\e} U(x) + P_\e(x)$ where $P_\e(x)$ is same in lemma \ref{lem-bar}. Then, we can show that $h^+$ is a super solution for large $K$. Hece $U$ is bounded uniformly on $e$ since $h^+$ is bounded uniformly on $\e$.
\end{proof}

\begin{remark}
$U$ is also bounded when $\Omega$ satisfies the uniform exterior sphere condition. And the proof is similar to the proof of case $\Omega$ is convex. 
\end{remark}

\begin{lem}[$\e$-Flatness] \label{lem-disc-flat}
Let $u^{\e}$ be the viscosity solution of \eqref{eqn-disc}. And suppose all the conditins in lemma \ref{lem-disc-grad} are satisfied. Then, there is a constant $C>0$ which is independent of $\e$ satisfying
\begin{equation*}
|u_{\e}(x_1)-u_{\e}(x_2)|\leq\, C \e
\end{equation*}
for every $x_1$, $x_2$ contained in a same $\e$-cell of $T_\e \cap \Omega$.
\end{lem}
\begin{proof}

Let $\widetilde{u}_\e(y) = u_\e(\e y)$. Then $\widetilde{u}_\e$ satisfies
\begin{equation*} \begin{cases}
a_{ij}(y) D_{ij} \widetilde{u}_\e(y) = \e^2 [f(\e y,y) - c(\widetilde{u}_\e(\e y),y)] &\text{ in } \frac{1}{\e} \Omega \setminus T^{\mathfrak a} \\
b^i(y) D_i \widetilde{u}_\e(y) = 0 &\text{ on } \frac{1}{\e}\Omega \bigcap \partial T^{\mathfrak a} \\
\widetilde{u}_\e = 0 &\text{ on } \partial \frac{1}{\e}\Omega \setminus T^{\mathfrak a}.
\end{cases} \end{equation*}

Let $Q$ be an unit cell with center 0 in $\R^n$ and let $B_a=B_a(0)$ be the intersection between $Q$ and $T^{\mathfrak a}$.
If $Q \cap \partial \left( \frac{1}{\e}\Omega \right)$ is nonempty, then we can prove the lemma by the barriers $\pm h_\e$ in lemma \ref{lem-disc-grad}. So, we will assume that $Q \subset \left( \frac{1}{\e}\Omega \right) \setminus T^{\mathfrak a}$ and $Q \cap \partial \left( \frac{1}{\e}\Omega \right) = \emptyset$.

Let $Q_1 = \left( -\displaystyle\frac{3}{2} , \frac{3}{2} \right) ^n $ be a cuve in $\R^n$. We note that $Q \subset Q_1$ and $Q_1 \setminus \overline{B}_a \subset \R^n \setminus T^{\mathfrak a}$ for small $a$. 
First, let $\widetilde{u}^* = \widetilde{u}_\e - \inf{Q_1} \widetilde{u}_\e$ then,
\begin{equation*} \begin{cases}
a_{ij}D_{ij} \widetilde{u}^* = \e^2 (f(\e y,y) - c(\widetilde{u}_\e(\e y),y) &\text{ in } Q_1 \setminus T^{\mathfrak a} \\
b^i(y)D_i \widetilde{u}^*  = 0  &\text{ on } Q_1 \bigcap \partial T^{\mathfrak a} \\
\end{cases} \end{equation*}

Let $S_0 = \sup_{Q_1 \setminus T^{\mathfrak a}} \widetilde{u}^*$, $I_0 = \inf_{Q_1 \setminus T^{\mathfrak a}} \widetilde{u}^* =0$, $S_1 = \sup_{Q \setminus T^{\mathfrak a}} \widetilde{u}^*$ and $I_1 = \inf_{Q \setminus T^{\mathfrak a}} \widetilde{u}^* =0$.

since we know that $u_\e$ is bounded, $f_\e = \e (f(\e y,y) - c(\widetilde{u}_\e(\e y),y)$ is bounded and small if $\e$ is small enough.
From the similar reason to lemma \ref{lem-precom-osc}, we have
\begin{equation*}
S_1 \le C(n,\lambda,\Lambda) (I_1 + \e \|f_\e\|_\infty).
\end{equation*}

So, if $\e \le \e_0$ for some $\e_0$, then we have
\begin{equation*}
S_1 \le C(n,\lambda,\Lambda) (I_1 + \e).
\end{equation*}

And from the discrete estimate, we have
\begin{equation*}
|\widetilde{u}^*(x) - \widetilde{u}^*(y)| = |u_\e (\e x) - u_\e (\e y)| \le C \e.
\end{equation*}
From this, we have
\begin{equation*} \begin{aligned}
S_0 \le S_1 + C \e \\
0 = I_0 \ge I_1 - C \e.
\end{aligned} \end{equation*}

Combining these result, we have
\begin{equation*}
S_1 \le C \e
\end{equation*}
where $C$ is a constant which is uniform on $\e$.
\end{proof}

The $\e$-Flatness and Discrete Gradient Estimate will give us Global  Lipschitz Estimate with  $\e$-error.
\begin{thm}[Global $\e$-Lipschitz Estimate] \label{thm-disc-global}
There is uniform constants $C>0$ such that
\begin{equation*}
|u_{\e}(x)-u_{\e}(y)|\leq\, C(|x-y|+\e)
\end{equation*}
for $x,y\in  \Omega_{\e}$.
\end{thm}

\noindent{\bf Acknowledgement.} Ki-Ahm Lee was supported by Basic
Science Research Program through the National Research Foundation of
Korea(NRF)  grant funded by the Korea
government(MEST)(2010-0001985).

\end{document}